\documentclass[a4paper,reqno]{amsart}

\usepackage{amssymb,upref}
\usepackage[mathcal]{euscript}

\newcommand{\bydef}{:=}

\newcommand{\id}{\mathrm{id}}

\newcommand{\espan}[1]{\mathrm{span}\left\{#1\right\}}
\newcommand{\matr}[1]{\left(\begin{smallmatrix}#1\end{smallmatrix}\right)}
\newcommand{\diag}{\mathrm{diag}}

\newcommand{\bi}{\mathbf{i}}
\newcommand{\bj}{\mathbf{j}}
\newcommand{\bl}{\mathbf{l}}

\newcommand{\cA}{\mathcal{A}}
\newcommand{\cB}{\mathcal{B}}
\newcommand{\cC}{\mathcal{C}}
\newcommand{\cD}{\mathcal{D}}

\newcommand{\cH}{\mathcal{H}}

\newcommand{\cJ}{\mathcal{J}}
\newcommand{\cK}{\mathcal{K}}
\newcommand{\cL}{\mathcal{L}}

\newcommand{\cO}{\mathcal{O}}
\newcommand{\cQ}{\mathcal{Q}}
\newcommand{\cR}{\mathcal{R}}

\newcommand{\cT}{\mathcal{T}}

\newcommand{\cX}{\mathcal{X}}

\newcommand{\frg}{{\mathfrak g}}

\newcommand{\frf}{{\mathfrak f}}
\newcommand{\fre}{{\mathfrak e}}

\newcommand{\frs}{{\mathfrak s}}

\newcommand{\NN}{\mathbb{N}}
\newcommand{\ZZ}{\mathbb{Z}}

\newcommand{\RR}{\mathbb{R}}
\newcommand{\CC}{\mathbb{C}}
\newcommand{\HH}{\mathbb{H}}
\newcommand{\OO}{\mathbb{O}}
\newcommand{\FF}{\mathbb{F}}

\newcommand{\charac}{\mathrm{char}}

\DeclareMathOperator{\End}{\mathrm{End}}

\DeclareMathOperator{\im}{\mathrm{im}\,}
\DeclareMathOperator{\Aut}{\mathrm{Aut}}
\DeclareMathOperator{\Diag}{\mathrm{Diag}}
\DeclareMathOperator{\Der}{\mathrm{Der}}
\DeclareMathOperator{\supp}{\mathrm{Supp}\,}
\DeclareMathOperator{\alg}{\mathrm{alg}}

\newcommand{\Ad}{\mathrm{Ad}}

\newcommand{\frsl}{{\mathfrak{sl}}}

\newcommand{\frso}{{\mathfrak{so}}}

\newcommand{\frsu}{{\mathfrak{su}}}
\newcommand{\frstu}{{\mathfrak{stu}}}

\newcommand{\tri}{\mathfrak{tri}}

\newcommand{\GL}{\mathrm{GL}}
\newcommand{\SL}{\mathrm{SL}}
\newcommand{\PSL}{\mathrm{PSL}}

\newcommand{\SO}{\mathrm{SO}}
\newcommand{\PSO}{\mathrm{PSO}}

\newcommand{\Spin}{\mathrm{Spin}}

\newcommand{\Br}{\mathrm{Br}}

\newcommand{\subo}{_{\bar 0}}
\newcommand{\subuno}{_{\bar 1}}

\newtheorem{theorem}{Theorem}
\newtheorem{proposition}[theorem]{Proposition}
\newtheorem{lemma}[theorem]{Lemma}
\newtheorem{corollary}[theorem]{Corollary}

\theoremstyle{definition}

\newtheorem{example}[theorem]{Example}
\newtheorem{remark}[theorem]{Remark}

\newenvironment{romanenumerate}
 {\begin{enumerate}
 
 }{\end{enumerate}}

\begin{document}

\title{Maximal finite abelian subgroups of $E_8$}

\author[C.~Draper]{Cristina Draper${}^\star$}
\address{Departamento de Matem\'atica Aplicada, Escuela de las Ingenier\'{\i}as, Universidad de M\'alaga, Ampliaci\'on Campus de Teatinos, 29071 M\'alaga, Spain}
\email{cdf@uma.es}
\thanks{${}^\star$ Supported by the Spanish Ministerio de Econom\'{\i}a y Competitividad---Fondo Europeo de Desarrollo Regional (FEDER) MTM 2010--15223 and by the Junta de Andaluc\'{\i}a grants FQM-336, FQM-1215 and FQM-246}

\author[A.~Elduque]{Alberto Elduque${}^\dagger$}
\address{Departamento de Matem\'{a}ticas
 e Instituto Universitario de Matem\'aticas y Aplicaciones,
 Universidad de Zaragoza, 50009 Zaragoza, Spain}
\email{elduque@unizar.es}
\thanks{${}^\dagger$ Supported by the Spanish Ministerio de Econom\'{\i}a y Competitividad---Fondo Europeo de Desarrollo Regional (FEDER) MTM2010-18370-C04-02 and by the Diputaci\'on General de Arag\'on---Fondo Social Europeo (Grupo de Investigaci\'on de \'Algebra)}



\date{}

\begin{abstract}
The maximal finite abelian subgroups, up to conjugation, of the simple algebraic group of type $E_8$ over an algebraically closed field of characteristic $0$ are computed. This is equivalent to the determination of the fine gradings on the simple Lie algebra of type $E_8$ with trivial neutral homogeneous component.
\end{abstract}

\maketitle

\setlength{\unitlength}{1mm}


\section{Introduction}

A systematic study of the gradings by abelian groups on the simple Lie algebras was initiated by Patera and Zassenhaus in \cite{PZ}. For the classical simple Lie algebra over an algebraically closed field of characteristic $0$, the fine gradings were classified in \cite{EldFineClassical}, for the exceptional simple algebras they were classified in \cite{DM_G2} and \cite{BT_G2} for $G_2$, in \cite{DM_F4} for $F_4$ (see also \cite{EK} and \cite{Cris_NonCompu}) and in \cite{DV} for $E_6$. The recent monograph \cite{EK13} collects, among other things, all these results and extensions to prime characteristic.

The problem of the classification of fine gradings, up to equivalence, on a simple Lie algebra over an algebraically closed field of characteristic $0$ contains in particular the problem of classifying the maximal finite abelian subgroups of the group of automorphisms of the algebra, up to conjugation.

The goal of this paper is the solution to this problem for the simple Lie algebra $\fre_8$ of type $E_8$, whose group of automorphisms is the exceptional simple algebraic group of type $E_8$.

\smallskip

A related problem is considered in \cite{Yu}, where the abelian subgroups $F$ of the compact (real) simple Lie groups $G$ satisfying the condition
\[
\dim\frg_0^F=\dim F,
\]
where $\frg_0$ is the Lie algebra of $G$ and $\frg_0^F$ is the subalgebra of fixed elements by the action of $F$, are studied.
This class of abelian subgroups present nice functorial properties exploited in \cite{Yu}, and it comprises the class of the maximal finite abelian subgroups ($\dim F=0$). The close relationship between compact Lie groups and complex reductive linear algebraic groups allows, in principle, to extract from \cite{Yu} the list of the maximal finite abelian subgroups of a simple linear algebraic group over $\CC$. However, some of the arguments in \cite{Yu} are not yet complete and this task is not easy.

Our approach works over arbitrary algebraically closed fields of characteristic $0$ and uses recent results on gradings on simple Lie algebras. We believe it has an independent interest.

\smallskip

The main result of the paper may be summarized in the following theorem.

\begin{theorem}\label{th:main_result}
Let $\FF$ be an algebraically closed field of characteristic $0$. Then, up to conjugation, the list of maximal finite abelian subgroups of the exceptional simple Lie group of type $E_8$ consists of:
\begin{romanenumerate}
\item Four elementary abelian groups, isomorphic to $\ZZ_2^9$, $\ZZ_2^8$, $\ZZ_3^5$ and $\ZZ_5^3$.
\item Three more subgroups, isomorphic to $\ZZ_6^3$, $\ZZ_4^3\times\ZZ_2^2$ and $\ZZ_4\times\ZZ_2^6$.
\end{romanenumerate}
\end{theorem}

The maximal elementary abelian $p$-subgroups of the algebraic groups have been obtained in \cite{Griess}, so the goal of this paper is to show the existence and uniqueness of the subgroups in item (ii) in the Theorem.

\medskip

Let us start with some definitions.

Let $A$ be an abelian group, an \emph{$A$-grading} on a nonassociative  (i.e., not necessarily associative) algebra $\cA$  over a field $\FF$ is a vector space decomposition
\[
\Gamma:\cA=\bigoplus_{a\in A} \cA_a
\]
such that $\cA_a \cA_b\subset \cA_{ab}$ for all $a,b\in A$.
If such a decomposition is fixed, we will refer to $\cA$ as an \emph{$A$-graded algebra}. The subspaces $\cA_a$ are said to be the \emph{homogeneous components} of $\Gamma$ and
the nonzero elements $x\in\cA_a$ are called {\em homogeneous of degree $a$}; we will write $\deg x=a$.
The {\em support} of $\Gamma$ is the set $\supp\Gamma\bydef\{a\in A : \cA_a\neq 0\}$.

Let
\[
\Gamma: \cA=\bigoplus_{a\in A} \cA_a\quad\text{and}\quad\Gamma':\cB=\bigoplus_{b\in B} \cB_b
\]
be two gradings on algebras, with supports $S$ and $T$, respectively.
$\Gamma$ and $\Gamma'$ are said to be \emph{equivalent} if there exists an isomorphism of algebras $\psi\colon\cA\to\cB$ and a bijection $\alpha\colon S\to T$ such that $\psi(\cA_s)=\cB_{\alpha(s)}$ for all $s\in S$. Any such $\psi$ will be called an {\em equivalence} of $\Gamma$ and $\Gamma'$.

Given a group grading $\Gamma$ on an algebra $\cA$, there are many groups $A$ such that $\Gamma$, regarded as a decomposition into a direct sum of subspaces such that the product of any two of them lies in a third one, can be realized as an $A$-grading, but there is one distinguished group among them \cite{PZ}.
We will say that an abelian group $U$ is a \emph{universal group of $\Gamma$} if, for any other realization of $\Gamma$ as an $A$-grading, there exists a unique homomorphism $U\to A$ that restricts to identity on $\supp\Gamma$.

One shows that the universal group, which we denote by $U(\Gamma)$, exists and depends, up to isomorphism, only on the equivalence class of $\Gamma$. Indeed, $U(\Gamma)$ is generated by $S=\supp\Gamma$ with defining relations $s_1s_2=s_3$ whenever $0\ne\cA_{s_1}\cA_{s_2}\subset\cA_{s_3}$ ($s_i\in S$).

Given gradings $\Gamma:\cA=\bigoplus_{a\in A}\cA_a$ and $\Gamma'\,\colon\,\cA=\bigoplus_{b\in B}\cA'_b$, we say that $\Gamma'$ is a {\em coarsening} of $\Gamma$, or that $\Gamma$ is a {\em refinement} of $\Gamma'$, if for any $a\in A$ there exists $b\in B$ such that $\cA_a\subset\cA'_b$. The coarsening (or refinement) is said to be {\em proper} if the inclusion is proper for some $a\in \supp\Gamma$.  A grading $\Gamma$ is said to be {\em fine} if it does not admit a proper refinement.

Over algebraically closed fields of characteristic zero, the classification of fine gradings on $\cA$ up to equivalence is the same as the classification of maximal diagonalizable subgroups (i.e., maximal quasitori) of $\Aut(\cA)$ up to conjugation (see e.g. \cite{PZ}). More precisely, given a grading $\Gamma$ on the algebra $\cA$ by the group $A$, let $\hat A$ be its group of characters (homomorphisms $A\rightarrow \FF^\times$). Any $\chi\in \hat A$ acts as an automorphism of $\cA$ by means of $\chi.x=\chi(a)x$ for any $a\in A$ and $x\in \cA_a$. In case $A$ is the universal group of $\Gamma$, this allows us to identify $\hat A$ with a quasitorus (the direct product of a torus and a finite subgroup) of the algebraic group $\Aut(\cA)$. This quasitorus is the subgroup $\Diag(\Gamma)$ consisting of the automorphisms $\varphi$ of $\cA$ such that the restriction of $\varphi$ to any homogeneous component is the multiplication by a (nonzero) scalar. (See \cite[\S 1.4]{EK13}.)

Conversely, given a quasitorus $Q$ of $\Aut(\cA)$, we can identify $Q$ with the group of characters $\hat A$ for $A$ the group of homomorphisms (as algebraic groups) $Q\rightarrow \FF^\times$. Then $Q$ induces an $A$-grading of $\cA$, where $\cA_a=\{x\in \cA: \chi(x)=a(\chi)x\ \forall\chi\in Q\}$ for any $a\in A$.  In this way \cite{PZ} the fine gradings on $\cA$, up to equivalence, correspond to the conjugacy classes in $\Aut(\cA)$ of the maximal quasitori (or maximal abelian diagonalizable subgroups) of $\Aut(\cA)$.

\smallskip

Let $\frg$ be a finite dimensional simple Lie algebra over an algebraically closed field $\FF$ of characteristic $0$, let $Q$ be a maximal quasitorus of $\Aut(\frg)$ and let $\Gamma$ be the associated fine grading. Then $Q=T\times F$ for a torus $T$ and a finite subgroup $F$.

If $T\ne 0$ or, equivalently, if the free rank of the universal group of $\Gamma$ is infinite, then $\Gamma$ induces a grading by a not necessarily reduced root system \cite{Eld_Fine} and it is determined by a fine grading on the coordinate algebra of the grading by the root system. Associative, alternative, Jordan or structurable algebras appear as coordinate algebras. In a sense, the classification of the fine gradings whose associated quasitori are not finite is reduced to the classification of some fine gradings on certain nonassociative algebras.

We are left then with the case in which $Q$ is a maximal finite abelian subgroup. In this case ($T=0$), the neutral homogeneous component of the associated grading $\Gamma$ (i.e., the subalgebra of the elements fixed by the automorphisms in $Q$) is trivial (see, for instance, \cite[Proposition 4.1]{Eld_Fine} or \cite[Corollary 5]{DM_F4}).

Hence, the goal of this paper is equivalent to the classification, up to equivalence, of the fine gradings on the simple Lie algebra $\fre_8$ with trivial neutral homogeneous component.

\smallskip

The next section will be devoted to survey some results on division gradings on matrix algebras, and the concept of Brauer invariant of an irreducible module for a graded semisimple Lie algebra. Section \ref{se:finite_order} will present some preliminary results on maximal finite quasitorus on simply connected algebraic groups, and will show that any maximal finite quasitorus $Q$ of $E_8=\Aut(\fre_8)$ is either $p$-elementary abelian for $p=2$, $3$ or $5$, or its exponent is $6$ or $4$ and it contains a specific automorphism of $\fre_8$ of order $6$ or $4$ (two possibilities here).

The following sections will deal with the different possibilities, assuming $Q$ is not $p$-elementary abelian.

It turns out that the corresponding fine gradings on $\fre_8$ are nicely described in terms of some nonassociative algebras. This will be reviewed in the last section.

\smallskip

From now on, \emph{the ground field $\FF$ will be assumed to be algebraically closed of characteristic zero}. Unadorned tensor products will indicate products over $\FF$.

\bigskip


\section{Division gradings. The Brauer invariant}\label{se:division_grading}

The aim of this section is to survey some well-known results on gradings on algebras of matrices $M_n(\FF)$ and on the induced gradings on the special Lie algebras $\frsl_n(\FF)$, in a way suitable for our purposes. The reader may consult \cite[Chapter~2]{EK13} and the references therein.

These results lead to the notion of the Brauer invariant of an irreducible module for a semisimple Lie algebra (see \cite{EKpr}).

\smallskip

Let $V$ be a vector space over $\FF$, $\dim V\geq 2$ and let $x,y\in\End_\FF(V)$ and $n\in\NN$ such that $x^n\in\FF^\times 1_V$, $y^n\in\FF^\times 1_V$ and $xy=\xi yx$, where $\xi$ is a primitive $n^\text{th}$ root of unity and $1_V$ denotes the identity map on $V$.

Then the elements $x^iy^j$, $0\leq i,j\leq n-1$ are eigenvectors for $\Ad_x$ and $\Ad_y$ ($\Ad_a(b)\bydef aba^{-1}$) with different eigenvalues, so they are linearly independent and the subalgebra $S$ generated by $x$ and $y$: $S=\alg\langle x,y\rangle$, is a simple unital subalgebra of $\End_\FF(V)$ of dimension $n^2$, as any ideal of $S$ is invariant under $\Ad_x$ and $\Ad_y$, so it contains an invertible element $x^iy^j$.

By the Double Centralizer Theorem, the centralizer $C=C_{\End_\FF(V)}(S)$ is simple too and $\End_\FF(V)$ is isomorphic to the tensor product $S\otimes C$. As a module for $S$, $V$ is the tensor product $V=V_1\otimes V_2$, where $V_1$ is the $n$-dimensional irreducible module for $S$, and $S$ acts trivially on $V_2$. Thus $S$ can be identified with $\End_\FF(V_1)$ and $C$ with $\End_\FF(V_2)$.

Multiplying by suitable scalars we may assume $x,y\in\SL(V_1)$ (the special linear group) and hence $x^n=y^n=1$ for odd $n$, or $x^n=y^n=-1$ for even $n$ (the characteristic polynomial of $x$ as an endomorphism of $V_1$ is $\det(\lambda 1_{V_1}-x)=\lambda^n+(-1)^n\det(x)=\lambda^n+(-1)^n$). Take an eigenvector $v\in V_1$ of $x$ with eigenvalue $1$ if $n$ is odd, or eigenvalue $\omega$, with $\omega^2=\xi$ if $n$ is even. In the basis $\{v,y(v),\ldots,y^{n-1}(v)\}$, the matrices of $x$ and $y$ are:
\[
x\leftrightarrow \begin{pmatrix} 1&0&0&\hdots&0\\ 0&\xi&0&\hdots&0\\
0&0&\xi^2&\hdots&0\\
\vdots&\vdots&\vdots&\ddots&\vdots\\ 0&0&0&\hdots&\xi^{n-1}\end{pmatrix},\qquad
y\leftrightarrow \begin{pmatrix} 0&0&\hdots&0&1\\ 1&0&\hdots&0&0\\ 0&1&\hdots&0&0\\
\vdots&\vdots&\ddots&\vdots&\vdots\\ 0&0&\hdots&1&0\end{pmatrix},
\]
for odd $n$, or
\[
x\leftrightarrow \begin{pmatrix} \omega&0&0&\hdots&0\\ 0&\omega^3&0&\hdots&0\\
0&0&\omega^5&\hdots&0\\
\vdots&\vdots&\vdots&\ddots&\vdots\\ 0&0&0&\hdots&\omega^{2n-1}\end{pmatrix},\qquad
y\leftrightarrow \begin{pmatrix} 0&0&\hdots&0&-1\\ 1&0&\hdots&0&0\\ 0&1&\hdots&0&0\\
\vdots&\vdots&\ddots&\vdots&\vdots\\ 0&0&\hdots&1&0\end{pmatrix},
\]
for even $n$. In particular, $x$ and $y$ are unique up to simultaneous conjugation. Moreover, the subgroup of $\Aut\bigl(\End_\FF(V)\bigr)\simeq \PSL(V)$ generated by the commuting automorphisms $\Ad_x$ and $\Ad_y$ is isomorphic to $\ZZ_n^2$ ($\ZZ_n\bydef \ZZ/n\ZZ$).

Write $\cR=\End_\FF(V)$, $\dim V\geq 2$. A grading $\Gamma: \cR=\bigoplus_{a\in A}\cR_a$ by an abelian group $A$ is said to be a \emph{division grading} (and $\cR$ is called a \emph{graded division algebra}), if any nonzero homogeneous element is invertible. This is equivalent to $\dim \cR_e=1$ (see, for instance, \cite[Lemma 2.20]{EK13}). In this case, there is a decomposition $V=V_1\otimes\cdots\otimes V_r$, with $\dim V_i=l_i\geq 2$ for any $i$, such that the support of $\Gamma$ is a subgroup $T$ of $A$ isomorphic to $\ZZ_{l_1}^2\times\cdots\times \ZZ_{l_r}^2$, and there are elements $x_i,y_i\in \SL(V_i)$ with $x_iy_i=\xi_i y_ix_i$, where $\xi_i$ is a fixed primitive $l_i^\text{th}$ root of $1$ such that for $t=(a_1,b_1,\ldots,a_r,b_r)\in \ZZ_{l_1}^2\times\cdots\times \ZZ_{l_r}^2\simeq T$, the homogeneous component of degree $t$ is spanned by
\[
X_t\bydef x_1^{a_1}y_1^{b_1}\otimes\cdots\otimes x_r^{a_r}y_r^{b_r}\in\End_\FF(V_1)\otimes\cdots\otimes\End_\FF(V_r)\simeq\End_\FF(V).
\]
This shows that $\cR$ is a twisted group algebra $\FF^\sigma T$ for a
suitable cocycle. Note that for $t,s\in T$, $X_tX_s=\beta(t,s)X_sX_t$, where $\beta:T\times T\rightarrow \FF^\times$ is the nondegenerate alternating bicharacter such that $\beta(t_i,s_i)=\xi_i$, $\beta(t_i,t_j)=\beta(t_i,s_j)=\beta(s_i,s_j)=1$ for $i\ne j$, where $t_i=\deg(x_i)$, $s_i=\deg(y_i)$ for any $i$. The pair $(T,\beta)$ determines the graded division algebra $\cR$, up to graded isomorphism. (See \cite[\S 2.2]{EK13}.)

\begin{remark}\label{re:Xt}
This grading is obviously fine and the associated maximal quasitorus in $\Aut(\End_\FF(V))\simeq \PSL(V)$ consists of the automorphisms $\Ad_{X_t}$, $t\in T$. (See \cite[Theorem 2.15 and Proposition 2.18]{EK13}.)
\end{remark}

\begin{remark}\label{re:division_grading}
Given a quasitorus $Q$ of $\Aut(\End_\FF(V))\simeq \PSL(V)$, take any $p\in\SL(V)$ such that $\Ad_p\in Q$. Then for any $q\in \SL(V)$ such that $\Ad_q\in Q$, $\Ad_q\Ad_p=\Ad_p\Ad_q$, since $Q$ is abelian, and hence $qpq^{-1}p^{-1}\in\FF^\times 1_V$. Hence $\Ad_q(p)\in\FF^\times p$, and thus $p$ is an eigenvector of the elements in $Q$, i.e., $p$ is a homogeneous element of the grading induced by $Q$.
\end{remark}

\smallskip

The gradings $\Gamma$ by abelian groups on the simple special Lie algebra $\frsl(V)$ are divided into two types depending on the quasitorus $\Diag(\Gamma)$ being contained in the connected component $\Aut(\frsl(V))^{\circ} \simeq \PSL(V)$ (type I) or not (type II). Type I gradings are all restrictions to $\frsl(V)$ of gradings on $\End_\FF(V)$. The gradings of type I with trivial neutral component thus correspond to the gradings on $\cR=\End_\FF(V)$ with $\dim\cR_e=1$, i.e., to the division gradings on $\End_\FF(V)$. Therefore we have the following result:

\begin{theorem}\label{th:slV}
Let $Q$ be a quasitorus of $\PSL(V)\,\bigl(\simeq\Aut(\frsl(V))^{\circ}\bigr)$ such that the neutral component of the induced grading on $\frsl(V)$ is trivial. Then there is a decomposition $V=V_1\otimes\cdots\otimes V_r$, ($\dim V_i=l_i\geq 2$ for any $i$) and elements $x_i,y_i\in\SL(V_i)$ with $x_iy_i=\xi_iy_ix_i$, where $\xi_i$ is a fixed primitive $l_i^\text{th}$ root of $1$, such that
\[
Q=\langle [x_1],[y_1],\ldots,[x_r],[y_r]\rangle\,
\bigl(\cong \ZZ_{l_1}^2\times\cdots\times\ZZ_{l_r}^2\bigr),
\]
where any endomorphism $z_i\in\GL(V_i)$ is identified with the endomorphism $1_{V_1}\otimes\cdots\otimes z_i\otimes\cdots\otimes 1_{V_l}$ (Kronecker product) in $\GL(V)$, and where $[z]$ denotes the class of $z\in\SL(V)$ in $\PSL(V)$.
\end{theorem}

\begin{example} The matrix algebra $M_8(\FF)$ is a graded division algebra in three nonequivalent ways, corresponding to $T=\ZZ_8^2$, $T=\ZZ_4^2\times\ZZ_2^2$, or $T=\ZZ_2^6$. This corresponds to considering $\cR$ as either $\End_\FF(V)$, $\End_\FF(V_1\otimes V_2)$ or $\End_\FF(U_1\otimes U_2\otimes U_3)$, with $\dim V=8$, $\dim V_1=4$, $\dim V_2=\dim U_i=2$, $i=1,2,3$.
\end{example}

\smallskip

Let $\cR$ be the matrix algebra $M_n(\FF)$, and let $\cR=\bigoplus_{a\in A}\cR_a$ be a grading on $\cR$ by an abelian group $A$. Then $\cR$ is isomorphic to $\End_\cD(W)$ where $\cD$ is a graded division algebra and $W$ is a finite dimensional graded right free module over $\cD$. Then $W$ is, up to isomorphism and shift of the grading, the only graded-simple module for $\cR$ (i.e., $W$ is $A$-graded with $\cR_aW_b\subseteq W_{ab}$ for any $a,b\in A$, and there is no proper graded submodule). Hence the isomorphism class of the graded division algebra $\cD$ is determined by the isomorphism class of the graded algebra $\cR$ and it is denoted by $[\cR]$ (so $[\cR]=[\cD]$).

Selecting a homogeneous $\cD$-basis $\{v_1,\ldots,v_k\}$ in $W$, $v_i\in W_{g_i}$, and setting $\tilde W=\mathrm{span}_\FF\left\{v_1,\ldots,v_k\right\}$, we can write $\cR\simeq \cC\otimes\cD$, where $\cC=\End_\FF(\tilde W)$ is a matrix algebra with the induced \emph{elementary grading}, i.e., the grading induced by the grading on its simple module: for any $i,j$, $\deg E_{ij}=g_ig_j^{-1}$, where $E_{ij}$ is the matrix unit that takes $v_j$ to $v_i$ and $v_l$ to $0$ for $l\ne j$.

The arguments above show that the class $[\cR]$ is determined by a  pair $(T,\beta)$, where $T$ is a subgroup of $A$ endowed with a nondegenerate alternating bicharacter $\beta:T\times T\rightarrow \FF^\times$.

Moreover, given two $A$-graded matrix algebras $\cR_1=\End_{\cD_1}(W_1)$ and $\cR_2=\End_{\cD_2}(W_2)$, the tensor product $\cR_1\otimes \cR_2$ is again an $A$-graded matrix algebra. Writing $\cR_i=\cC_i\otimes\cD_i$, $i=1,2$, and $\cD_1\otimes\cD_2=\cC\otimes \cD$ we obtain $\cR_1\otimes\cR_2=\bigl(\cC_1\otimes\cC_2\otimes\cC\bigr)\otimes\cD$, where the first factor has an elementary grading and the second factor a division grading. Then $[\cR_1\otimes\cR_2]=[\cD]$, which depends only on $[\cD_1]$ and $[\cD_2]$. Thus we obtain an abelian group, called the \emph{$A$-graded Brauer group} (see \cite[\S 2]{EKpr}). The behavior of this group mimics the behavior of the classical Brauer group (but this latter one is trivial over algebraically closed fields!!). In particular we have $[\cD]^{-1}=[\cD^\text{op}]$, as $\cD\otimes\cD^\text{op}\cong\End_\FF(\cD)$ as graded algebras, and the grading on $\End_\FF(\cD)$ is elementary (induced by the grading on $\cD$).

\smallskip

Let now $\Gamma:\cL=\bigoplus_{a\in A}\cL_a$ be a grading by the abelian group $A$ of a finite dimensional semisimple Lie algebra $\cL$ and let $V$ be a finite dimensional irreducible module for $\cL$. Assume that the image of the group of characters $\hat A$ in $\Aut(\cL)$ lies in the connected component $\Aut(\cL)^{\circ}$ (i.e., it consists of inner automorphisms). Then \cite[\S 3]{EKpr} the matrix algebra $\End_\FF(V)$ is $A$-graded in a unique way satisfying that the associated surjective homomorphism of associative algebras $\rho:U(\cL)\rightarrow \End_\FF(V)$ ($U(\cL)$ denotes the universal enveloping algebra) is a homomorphism of $A$-graded algebras. Hence $\End_\FF(V)\cong\End_\cD(W)$ for some $\cD$ and $W$ as above, and we write $\Br(V)=[\cD]$. This is called the \emph{Brauer invariant} of the irreducible $\cL$-module $V$.

If $V$ admits an $A$-grading compatible with the action of $\cL$ (i.e., $\cL_aV_b\subseteq V_{ab}$ for any $a,b\in A$), then the (unique) $A$-grading on $\End_\FF(V)$ is induced by the grading on $V$ (it is elementary) and $\cD=\FF$, i.e., the Brauer invariant is trivial. This $A$-grading on $V$ is unique up to a shift. In particular, its homogeneous components are uniquely determined.

Later on, we will make use of the following result:

\begin{lemma}\label{le:important}
Let $\cL=\bigoplus_{\bar r\in\ZZ_n}\cL_{\bar r}$ be a $\ZZ_n$-graded finite dimensional semisimple Lie algebra and let $\Gamma:\cL=\bigoplus_{a\in A}\cL_a$ be a grading on $\cL$ by an abelian group $A$ that refines the $\ZZ_n$-grading (i.e., each nonzero homogeneous component $\cL_a$ is contained in a (unique) subspace $\cL_{\bar r}$, $\bar r\in\ZZ_n$). Assume that $\cL\subo$ is semisimple and each $\cL_{\bar r}$, $\bar r\ne \bar 0$, is an irreducible module for $\cL\subo$. Then $\Gamma$ is determined, up to equivalence, by its restriction $\Gamma\subo$ to $\cL\subo$.

In particular, if $\Gamma$ is fine, the maximal quasitorus $\Diag(\Gamma)$ is determined by $\Gamma\subo$.
\end{lemma}
\begin{proof}
Since each $\cL_{\bar r}$, $\bar r\ne \bar 0$, is irreducible, the restriction $\Gamma_{\bar r}$ of $\Gamma$ to $\cL_{\bar r}$ is the unique $A$-grading, up to a shift, on $\cL_{\bar r}$ compatible with the $A$-grading $\Gamma\subo$ on $\cL\subo$. Hence the homogeneous components of $\Gamma$ are all uniquely determined by $\Gamma\subo$.
\end{proof}

\bigskip


\section{Finite order automorphisms}\label{se:finite_order}

Let $\frg$ be a finite dimensional simple Lie algebra and let $G$ be its group of automorphisms: $G=\Aut(\frg)$. Given a subgroup $H$ of $G$, $C_G(H)$ will denote its centralizer in $G$.

The aim of this section is to show that the maximal finite quasitorus of $\Aut(\fre_8)$ are  either $p$-elementary abelian for $p=2$, $3$ or $5$, or their exponent is either $6$ or $4$ and they contain  specific automorphisms of $\fre_8$ of order $6$ or $4$.

\begin{lemma}\label{le:selfcentralizing}
Let $Q$ be a maximal quasitorus in $G$, then $Q$ is self-centralizing: $C_G(Q)=Q$.
\end{lemma}
\begin{proof}
By maximality $Q$ is a closed subgroup of the algebraic group $G$. For any $x\in C_G(Q)$, let $x=x_sx_n$ be its Jordan decomposition \cite[\S 15]{Hum}, then the closure of the subgroup generated by $Q$ and $x_s$ is diagonalizable, so by maximality of $Q$, $x_s\in Q$, and thus the quotient $C_G(Q)/Q$ is unipotent, and hence nilpotent \cite[\S 17.5]{Hum}. We conclude that $C_G(Q)$ is nilpotent, because $Q$ is central in $C_G(Q)$. Then \cite[III.3.4, Proposition 3.6]{GOV} implies that since $Q$ is reductive, so is $C_G(Q)$ and, therefore, $C_G(Q)$ is reductive and nilpotent, and hence its connected component satisfies $C_G(Q)^{\circ}=Z\bigl(C_g(Q)^{\circ})$, and it consists of semisimple elements. By maximality $C_G(Q)^{\circ}=Q^{\circ}$, so $[C_G(Q):Q]\leq [C_G(Q):C_G(Q)^{\circ}] <\infty$. Therefore $C_G(Q)/Q$ is unipotent and finite, so it is trivial (recall that we are assuming $\charac\FF=0$).
\end{proof}

In case $\frg=\fre_8$, the group $G$ is connected and simply connected, so the next result applies.

\begin{lemma}\label{le:gtheta_semisimple}
Assume that $G$ is semisimple, connected and simply connected, and let $Q$ be a maximal quasitorus of $G$ with $Q$ finite. Then for any $\theta\in Q$, the subalgebra of fixed elements $\frg^\theta\,\bigl(\bydef \{x\in\frg: \theta(x)=x\}\bigr)$ is a semisimple subalgebra.
\end{lemma}
\begin{proof}
Since $\theta$ is semisimple (finite order) and $G$ is connected and simply connected, $C_G(\theta)$ is reductive \cite[Theorem 2.2]{HumCC}, and since $G$ is simply connected, $C_G(\theta)$ is connected \cite[Theorem 2.11]{HumCC}. Then \cite[Lemma 19.5]{HumCC}, $Z\bigl(C_G(\theta)\bigr)^{\circ}$ is a torus.

But $Z\bigl(C_G(\theta)\bigr)^{\circ}$ is contained in any maximal quasitorus of $C_G(\theta)$, so it is contained in $Q$. Since $Q$ is finite, we get $Z\bigl(C_G(\theta)\bigr)^{\circ}=1$, so $\dim Z\bigl(C_G(\theta)\bigr)=0$ and hence the Lie algebra $\cL\bigl(Z\bigl(C_G(\theta)\bigr)\bigr)$ is trivial. It follows that the reductive Lie algebra $\frg^\theta=\cL\bigl(C_G(\theta)\bigr)$ has trivial center, so it is semisimple.
\end{proof}

The finite order automorphisms of $\frg$ are classified, up to conjugation, in \cite[\S 8.6]{Kac} in terms of affine Dynkin diagrams and sequences of relatively prime nonnegative integers $(s_0,\ldots,s_l)$ (here $l+1$ is the number of nodes in the affine Dynkin diagram). For $\frg=\fre_8$, if $\theta$ is a nontrivial finite order automorphisms and $\frg^\theta$ is semisimple, \cite[Proposition 8.6]{Kac} shows that the sequence is of the form $(0,\ldots,1,\ldots,0)$ with only one $s_i=1$ and $i>0$. That is, only one node is involved in the affine Dynkin diagram $E_8^{(1)}$:
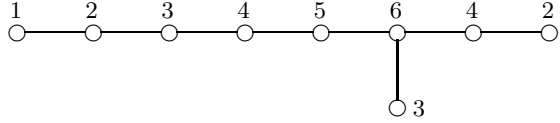
\[
\begin{picture}(80,20)
	\put(0,12){\circle{2}}
	\put(10,12){\circle{2}}
	\put(20,12){\circle{2}}
    \put(30,12){\circle{2}}
    \put(40,12){\circle{2}}
    \put(50,12){\circle{2}}
    \put(60,12){\circle{2}}
    \put(70,12){\circle{2}}
    \put(50,2){\circle{2}}
	
    \put(1,12){\line(1,0){8}}
    \put(11,12){\line(1,0){8}}
    \put(21,12){\line(1,0){8}}
    \put(31,12){\line(1,0){8}}
    \put(41,12){\line(1,0){8}}
    \put(51,12){\line(1,0){8}}
    \put(61,12){\line(1,0){8}}
    \put(50,3){\line(0,1){8}}
	
	\put(-1,14){\small 1}
	\put(9,14){\small 2}
	\put(19,14){\small 3}
    \put(29,14){\small 4}
    \put(39,14){\small 5}
    \put(49,14){\small 6}
    \put(59,14){\small 4}
    \put(69,14){\small 2}
    \put(52,1){\small 3}
	
\end{picture}
\]
This unique node will be highlighted putting it in black. Thus, for instance, the diagram
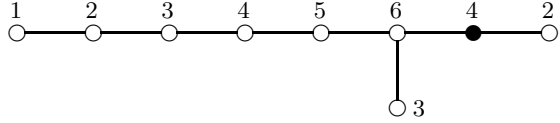
\[
\begin{picture}(80,20)
	\put(0,12){\circle{2}}
	\put(10,12){\circle{2}}
	\put(20,12){\circle{2}}
    \put(30,12){\circle{2}}
    \put(40,12){\circle{2}}
    \put(50,12){\circle{2}}
    \put(60,12){\circle*{2}}
    \put(70,12){\circle{2}}
    \put(50,2){\circle{2}}
	
    \put(1,12){\line(1,0){8}}
    \put(11,12){\line(1,0){8}}
    \put(21,12){\line(1,0){8}}
    \put(31,12){\line(1,0){8}}
    \put(41,12){\line(1,0){8}}
    \put(51,12){\line(1,0){8}}
    \put(61,12){\line(1,0){8}}
    \put(50,3){\line(0,1){8}}
	
	\put(-1,14){\small 1}
	\put(9,14){\small 2}
	\put(19,14){\small 3}
    \put(29,14){\small 4}
    \put(39,14){\small 5}
    \put(49,14){\small 6}
    \put(59,14){\small 4}
    \put(69,14){\small 2}
    \put(52,1){\small 3}
	
\end{picture}
\]
represents an automorphism of order $4$.

Therefore, for $\frg=\fre_8$, any automorphism $\theta$ with $\frg^\theta$ semisimple has order at most $6$, and up to conjugation there are only
\begin{itemize}
\item two order $2$ automorphisms,
\item two order $3$ automorphisms,
\item two order $4$ automorphisms,
\item one order $5$ automorphism, and
\item one order $6$ automorphism,
\end{itemize}
with semisimple $\frg^\theta$.

\begin{proposition}\label{pr:important}
Let $Q$ be a maximal finite abelian subgroup of $\Aut(\fre_8)$. Then either $Q$ is $p$-elementary abelian with $p=2,3$ or $5$, or the exponent of $Q$ is either $4$ or $6$. Moreover, for any $\theta\in Q$, the fixed subalgebra $\fre_8^\theta$ is semisimple.
\end{proposition}
\begin{proof}
By Lemma \ref{le:gtheta_semisimple}, $\fre_8^\theta$ is semisimple for any $\theta\in Q$, so the order of $\theta$ is at most $6$. The result follows at once.
\end{proof}

The maximal elementary abelian subgroups of algebraic groups have been classified in \cite{Griess}. For $\Aut(\fre_8)$ there are only four such subgroups which coincide with its centralizer (and hence they are maximal abelian subgroups). They are isomorphic to $\ZZ_2^9$, $\ZZ_2^8$, $\ZZ_3^5$ and $\ZZ_5^3$.

Hence we must consider the situation in which the maximal finite abelian subgroup $Q$ of $\Aut(\fre_8)$ contains an automorphism of order $6$ or $4$. (It cannot contain both as there are no elements of order $12$ in $Q$.)

\bigskip


\section{Order $6$ automorphism}

We assume in this section that $Q$ is a maximal finite abelian subgroup of $G=\Aut(\frg)$, with $\frg=\fre_8$, and that $Q$ contains an automorphism $\theta$ of order $6$. As $\frg^\theta$ is semisimple (Lemma \ref{le:gtheta_semisimple}), $\theta$ is conjugate to the automorphism corresponding to the diagram
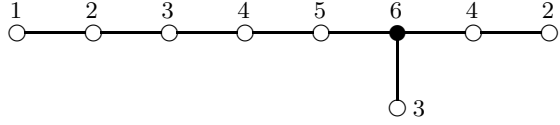
\[
\begin{picture}(80,20)
	\put(0,12){\circle{2}}
	\put(10,12){\circle{2}}
	\put(20,12){\circle{2}}
    \put(30,12){\circle{2}}
    \put(40,12){\circle{2}}
    \put(50,12){\circle*{2}}
    \put(60,12){\circle{2}}
    \put(70,12){\circle{2}}
    \put(50,2){\circle{2}}
	
    \put(1,12){\line(1,0){8}}
    \put(11,12){\line(1,0){8}}
    \put(21,12){\line(1,0){8}}
    \put(31,12){\line(1,0){8}}
    \put(41,12){\line(1,0){8}}
    \put(51,12){\line(1,0){8}}
    \put(61,12){\line(1,0){8}}
    \put(50,3){\line(0,1){8}}
	
	\put(-1,14){\small 1}
	\put(9,14){\small 2}
	\put(19,14){\small 3}
    \put(29,14){\small 4}
    \put(39,14){\small 5}
    \put(49,14){\small 6}
    \put(59,14){\small 4}
    \put(69,14){\small 2}
    \put(52,1){\small 3}
	
\end{picture}
\]
The automorphism $\theta$ thus induces a grading by $\ZZ_6$: $\frg=\bigoplus_{\bar r\in\ZZ_6}\frg_{\bar r}$, with $\frg_{\bar r}=\{x\in\frg: \theta(x)=\xi^r x\}$, where $\xi$ is a primitive $6^\text{th}$ root of $1$. Then, up to isomorphism, we have
\[
\begin{split}
\frg\subo&=\frsl(U)\oplus\frsl(V)\oplus\frsl(W),\\
\frg_{\bar 1}&=U\otimes V\otimes W,\\
\frg_{\bar 2}&=1\otimes V^*\otimes \wedge^2 W,\\
\frg_{\bar 3}&=U\otimes 1\otimes\wedge^3 W,\\
\frg_{\bar 4}&=1\otimes V\otimes \wedge^4 W,\\
\frg_{\bar 5}&=U\otimes V^*\otimes\wedge^5 W,
\end{split}
\]
where $U$, $V$ and $W$ are vector spaces of dimension $2$, $3$ and $6$ respectively. The expression above for $\frg_{\bar r}$, $\bar r\ne \bar 0$, gives the structure of $\frg_{\bar r}$ as a module for $\frg\subo$. For $\frg\subo$, $\frg_{\bar 1}$ and $\frg_{\bar 5}$ this follows from \cite[Proposition 8.6]{Kac}, as well as the fact that, for $\bar r\ne \bar 0$, $\frg_{\bar r}$ is an irreducible module for $\frg\subo$. The other components are computed easily.

In this case, the connected component $\Aut(\frg\subo)^{\circ}$ is isomorphic to $\PSL(U)\times\PSL(V)\times\PSL(W)$. Moreover, for any $(a,b,c)\in\SL(U)\times\SL(V)\times\SL(W)$, there is an automorphism $\phi_{a,b,c}$ of $\frg$ with
\[
\phi_{a,b,c}\vert_{\frg\subo}=\bigl(\Ad_a,\Ad_b,\Ad_c),\quad
\phi_{a,b,c}\vert_{\frg\subuno}=a\otimes b\otimes c.
\]
Note that $\frg\subuno$ generates $\frg$ as an algebra, so the automorphism $\phi_{a,b,c}$ is determined by its action on $\frg\subuno$.

We then have the following homomorphisms (of algebraic groups):
\[
\begin{split}
\Phi: \SL(U)\times\SL(V)\times\SL(W)&\longrightarrow C_G(\theta),\\
 (a,b,c)\qquad&\mapsto\quad \phi_{a,b,c},\\[10pt]
\Psi: C_G(\theta)&\longrightarrow \Aut(\frg\subo),\\
\varphi\quad&\mapsto\quad \varphi\vert_{\frg\subo}.
\end{split}
\]
Note that $Q$ is contained in $C_G(\theta)$ and that $\theta=\phi_{1_U,1_V,\xi 1_W}$.

The kernel and image of both $\Phi$ and $\Psi$ are computed next.

\begin{lemma}\label{le:6}
\begin{romanenumerate}
\item $\im\Psi=\Aut(\frg\subo)^{\circ}\,\bigl(\simeq\PSL(U)\times\PSL(V)\times\PSL(W)\bigr)$ and $\ker\Psi=\langle \theta\rangle$ (the subgroup generated by $\theta$).
\item $\Phi$ is surjective and $\ker\Phi=\langle (-1_U,\xi^2 1_V,\xi 1_W)\rangle$, which is a cyclic group of order $6$.
\end{romanenumerate}
\end{lemma}
\begin{proof}
If $\varphi$ is in $\ker\Psi$, then $\varphi\vert_{\frg\subo}=\id$, so by Schur's Lemma, $\varphi\vert_{\frg\subuno}=\lambda\id$ for a nonzero scalar $\lambda$. But $\frg\subuno$ generates $\frg$ and this forces $\lambda^6=1$, hence $\varphi$ is a power of $\theta$.

Since $C_G(\theta)$ is connected (proof of Lemma \ref{le:gtheta_semisimple}), $\im\Psi$ is contained in $\Aut(\frg\subo)^{\circ}$, so for any $\varphi\in C_G(\theta)$ there are elements $a\in\SL(U)$, $b\in\SL(V)$ and $c\in\SL(W)$ such that $\varphi\vert_{\frg\subo}=\bigl(\Ad_a,\Ad_b,\Ad_c\bigr)$. But $\phi_{a,b,c}\in C_G(\theta)$, and $\Psi(\varphi)=\Psi(\phi_{a,b,c})$. Hence $\varphi\phi_{a,b,c}^{-1}\in\ker\Psi=\langle\theta\rangle\subseteq \im\Phi$. Thus, $\varphi\in \im\Phi$ and $\Phi$ is onto.

But $\Psi(\im\Phi)$ fills $\Aut(\frg\subo)^{\circ}\simeq \PSL(U)\times\PSL(V)\times\PSL(W)$, so we obtain that $\im\Psi=\Aut(\frg\subo)^{\circ}$.

Finally, for any $a\in\SL(U)$, $b\in\SL(V)$ and $c\in\SL(W)$, the automorphism $\phi_{a,b,c}$ is the identity if and only if $a\otimes b\otimes c=\id$ in $\End_\FF(U\otimes V\otimes W)$, and this happens if and only if there are scalars $\lambda,\mu,\nu\in\FF^\times$ with $\lambda\mu\nu=1$ such that $a=\lambda 1_U$, $b=\mu 1_V$ and $c=\nu 1_W$ (which implies $\lambda^2=\mu^3=\nu^6=1$ because the determinant of these endomorphisms is $1$). This shows that $\ker \Phi$ is generated by $(-1_U,\xi^2 1_V,\xi 1_W)=(\xi^3 1_U,\xi^2 1_V,\xi 1_W)$.
\end{proof}

Denote by $\pi_U$, $\pi_V$ and $\pi_W$ the projections of $\Aut(\frg\subo)^{\circ}$ onto $\PSL(U)$, $\PSL(V)$ and $\PSL(W)$ respectively. The quasitorus $Q$ induces a fine grading $\Gamma$ on $\frg$ with trivial neutral homogeneous component, which restricts to a grading $\Gamma\subo$ on $\frg\subo$ and hence on $\frsl(U)$, $\frsl(V)$ and $\frsl(W)$ with trivial neutral homogeneous components. In particular, $\pi_U\circ\Psi(Q)$ is a diagonalizable subgroup in $\PSL(U)\,\bigl(\simeq\Aut(\frsl(U))^{\circ}\bigr)$ whose induced grading $\Gamma_U$ satisfies that its neutral component is trivial. The only possibility (Theorem \ref{th:slV}) is that $\pi_U\circ\Psi(Q)$ be isomorphic to $\ZZ_2^2$. In the same vein, $\pi_V\circ\Psi(Q)\cong \ZZ_3^2$ and $\pi_W\circ\Psi(Q)\cong \ZZ_6^2$.

This shows, in particular, that there are elements $c_1,c_2\in\SL(W)$ with $c_1^6=c_2^6=-1$ and $c_1c_2=\xi c_2c_1$ such that $\pi_W\circ\Psi(Q)=\langle [c_1],[c_2]\rangle$. Hence there are elements $a_1,a_2\in\SL(U)$ and $b_1,b_2\in\SL(V)$, such that $\phi_{a_1,b_1,c_1}$ and $\phi_{a_2,b_2,c_2}$ are in $Q$. Now we get
\begin{itemize}
\item Since $\pi_U\circ\Psi(Q)\cong \ZZ_2^2$, we have $\Ad_{a_i}^2=\id$, $i=1,2$, so $a_i^2=\epsilon_i 1_U$, with $\epsilon_i=\pm 1$, $i=1,2$ (as $\det(a_i)=1$).
\item Also, $\pi_V\circ\Psi(Q)\cong \ZZ_3^2$, so $\Ad_{b_i}^3=1$ and $b_i^3=\mu_i 1_V$ with $\mu_i^3=1$, $i=1,2$.
\item $\phi_{a_i,b_i,c_i}^6=\id$, so $\id=a_i^6\otimes b_i^6\otimes c_i^6=-a_i^6\otimes b_i^6\otimes 1_W=(-\epsilon_i^3\mu_i^2)\id$. Hence $\epsilon_i\mu_i^2=-1$, and this forces $\epsilon_i=-1$ and $\mu_i=1$, $i=1,2$.
\item Since $Q$ is abelian, $\phi_{a_1,b_1,c_1}\phi_{a_2,b_2,c_2}=\phi_{a_2,b_2,c_2}\phi_{a_1,b_1,c_1}$, and hence $a_1a_2\otimes b_1b_2\otimes c_1c_2=a_2a_1\otimes b_2b_1\otimes c_2c_1$, that is, $\xi a_1a_2\otimes b_1b_2=a_2a_1\otimes b_2b_1$. It then follows that there are scalars $\mu,\nu\in\FF^\times$ such that $a_1a_2=\mu a_2a_1$ and $b_1b_2=\nu b_2b_1$. Besides, $\mu^2=1$ (because $\det(a_1a_2)=1$) and $\nu^3=1$, and $\xi\mu\nu=1$. We conclude that $a_1a_2=-a_2a_1$ and $b_1b_2=\xi^2b_2b_1$ and hence we obtain $\pi_U\circ\Psi(Q)=\langle [a_1],[a_2]\rangle$ and $\pi_V\circ\Psi(Q)=\langle [b_1],[b_2]\rangle$
\end{itemize}

\begin{theorem}\label{th:6}
Under the conditions above, the maximal quasitorus $Q$ is isomorphic to $\ZZ_6^3$. Moreover, it is given explicitly by
\[
Q=\langle \phi_{a_1,b_1,c_1},\phi_{a_2,b_2,c_2},\theta\rangle
\]
for $a_1,a_2\in\SL(U)$ with $a_1^2=a_2^2=-1_U$, $a_1a_2=-a_2a_1$, $b_1,b_2\in\SL(V)$ with $b_1^3=b_2^3=1_V$, $b_1b_2=\xi^2b_2b_1$, and $c_1,c_2\in\SL(W)$ with $c_1^6=c_2^6=-1_W$, $c_1c_2=\xi c_2c_1$. ($\xi$ is a primitive $6^\text{th}$ root of $1$.)
\end{theorem}
\begin{proof}
If $\varphi\in Q$, then $\pi_W\circ\Psi(\varphi)\in \langle [c_1],[c_2]\rangle$, so there are integers $0\leq n_1,n_2\leq 5$ such that
\[
\varphi\phi_{a_1,b_1,c_1}^{n_1}\phi_{a_2,b_2,c_2}^{n_2}\in\ker\bigl(\pi_W\circ\Psi\bigr).
\]
Hence we may assume that $\varphi$ is in $\ker\bigl(\pi_W\circ\Psi\bigr)$. Since $\Phi$ is onto, there are elements $a\in\SL(U)$ and $b\in\SL(V)$ with $\varphi=\phi_{a,b,1_W}$. (Note that $\theta=\phi_{1_U,1_V,\xi 1_W}=\phi_{\xi^3 1_U,\xi^4 1_V,1_W}$.)

Since $\pi_U\circ\Psi(\varphi)$ lies in $\langle [a_1],[a_2]\rangle$, there is a scalar $\lambda\in\FF^\times$ and integers $0\leq r_1,r_2\leq 2$ with $a=\lambda a_1^{r_1}a_2^{r_2}$ and, similarly, there is a scalar $\mu\in\FF^\times$ and integers $0\leq s_1,s_2\leq 2$ with $b=\mu b_1^{s_1}b_2^{s_2}$. Also $\lambda^2=1=\mu^3$ because the determinants are always $1$.

Now, $Q$ is abelian, so $\varphi\phi_{a_1,b_1,c_1}=\phi_{a_1,b_1,c_1}\varphi$ or, equivalently,
\[
(a\otimes b\otimes 1_W)(a_1\otimes b_1\otimes c_1)
=(a_1\otimes b_1\otimes c_1)(a\otimes b\otimes 1_W),
\]
and, since $aa_1=(-1)^{r_2}a_1a$ and $bb_1=\xi^{2s_2}b_1b$, we obtain $(-1)^{r_2}\xi^{2s_2}=1$, which gives $r_2=s_2=0$. In the same vein we get $r_1=s_1=0$. Therefore $\varphi\in\ker\Psi=\langle\theta\rangle$, and we obtain $Q=\langle \phi_{a_1,b_1,c_1},\phi_{a_2,b_2,c_2},\theta\rangle$. Moreover, the three generators have order $6$, $\theta\in\ker\bigl(\pi_W\circ\Psi\bigr)$ and $\im\bigl(\pi_W\circ\Psi\vert_Q\bigr)=\langle [c_1],[c_2]\rangle\cong\ZZ_6^2$, so $Q$ is isomorphic to $\ZZ_6^3$.
\end{proof}

\begin{corollary}\label{co:6}
Up to conjugation, $\Aut(\fre_8)$ contains a unique maximal finite abelian subgroup with elements of order $6$.
\end{corollary}

\bigskip


\section{Order $4$ automorphism. Type I}

Assume now that our maximal finite abelian subgroup $Q$ of $G=\Aut(\frg)$, with $\frg=\fre_8$,  contains an automorphism $\theta$ of order $4$ which is conjugate to the automorphism corresponding to the diagram
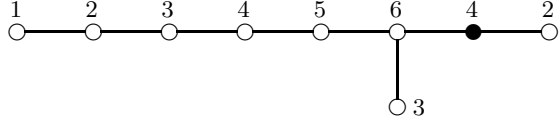
\[
\begin{picture}(80,20)
	\put(0,12){\circle{2}}
	\put(10,12){\circle{2}}
	\put(20,12){\circle{2}}
    \put(30,12){\circle{2}}
    \put(40,12){\circle{2}}
    \put(50,12){\circle{2}}
    \put(60,12){\circle*{2}}
    \put(70,12){\circle{2}}
    \put(50,2){\circle{2}}
	
    \put(1,12){\line(1,0){8}}
    \put(11,12){\line(1,0){8}}
    \put(21,12){\line(1,0){8}}
    \put(31,12){\line(1,0){8}}
    \put(41,12){\line(1,0){8}}
    \put(51,12){\line(1,0){8}}
    \put(61,12){\line(1,0){8}}
    \put(50,3){\line(0,1){8}}
	
	\put(-1,14){\small 1}
	\put(9,14){\small 2}
	\put(19,14){\small 3}
    \put(29,14){\small 4}
    \put(39,14){\small 5}
    \put(49,14){\small 6}
    \put(59,14){\small 4}
    \put(69,14){\small 2}
    \put(52,1){\small 3}
	
\end{picture}
\]
These order $4$ automorphisms will be said of \emph{type I}.

The automorphism $\theta$ thus induces a grading by $\ZZ_4$: $\frg=\bigoplus_{\bar r\in\ZZ_4}\frg_{\bar r}$, with $\frg_{\bar r}=\{x\in\frg: \theta(x)=\bi^r x\}$, where $\bi$ is a primitive $4^\text{th}$ root of $1$. Then, up to isomorphism, we have
\[
\begin{split}
\frg\subo&=\frsl(U)\oplus\frsl(V),\\
\frg_{\bar 1}&=U\otimes \wedge^2V,\\
\frg_{\bar 2}&=1\otimes  \wedge^4 V,\\
\frg_{\bar 3}&=U\otimes \wedge^6 V
\end{split}
\]
where $U$ and $V$ are vector spaces of dimension $2$ and $8$ respectively.

As in the previous section, we have group homomorphisms
\[
\begin{split}
\Phi: \SL(U)\times\SL(V)&\longrightarrow C_G(\theta),\\
 (a,b)\qquad&\mapsto\quad \phi_{a,b},\\[10pt]
\Psi: C_G(\theta)&\longrightarrow \Aut(\frg\subo),\\
 \varphi\quad&\mapsto\quad \varphi\vert_{\frg\subo},
\end{split}
\]
where $\phi_{a,b}\vert_{\frg\subuno}=a\otimes\wedge^2b$ and $\phi_{a,b}\vert_{\frg\subo}=(\Ad_a,\Ad_b)$.

Let $\omega$ be a primitive eighth root of $1$ with $\bi=\omega^2$. Then $\theta=\phi_{\bi 1_U,1_V}=\phi_{1_U,\omega 1_V}$. The next result is proved along the same lines as Lemma \ref{le:6}.

\begin{lemma}\label{le:4I}
\begin{romanenumerate}
\item $\im\Psi=\Aut(\frg\subo)^{\circ}\,\bigl(\simeq\PSL(U)\times\PSL(V)\bigr)$ and $\ker\Psi=\langle \theta\rangle$.
\item $\Phi$ is surjective and $\ker\Phi=\langle (-1_U,\omega^2 1_V)\rangle$, which is a cyclic group of order $4$.
\end{romanenumerate}
\end{lemma}

Denote by $\pi_U$ and $\pi_V$ the projections of $\Aut(\frg\subo)^{\circ}$ onto $\PSL(U)$ and $\PSL(V)$ respectively. Then $\pi_U\circ\Psi(Q)$ is necessarily isomorphic to $\ZZ_2^2$, while  $\pi_V\circ\Psi(Q)$, which is a $2$-group of exponent $\leq 4$, is isomorphic either to $\ZZ_4^2\times\ZZ_2^2$, or to $\ZZ_2^6$.

In particular, there are elements $a_1,a_2\in\SL(U)$ such that $a_1^2=-1_U=a_2^2$, $a_1a_2=-a_2a_1$ and $\pi_U\circ\Psi(Q)=\langle [a_1],[a_2]\rangle$, and hence there are elements $b_1,b_2\in\SL(V)$ such that $\phi_{a_1,b_1},\phi_{a_2,b_2}$ are in $Q$.

Since $Q$ is abelian,
\[
\phi_{a_1a_2,b_1b_2}=\phi_{a_1,b_1}\phi_{a_2,b_2}=
\phi_{a_2,b_2}\phi_{a_1,b_1}=\phi_{a_2a_1,b_2b_1}=\phi_{-a_1a_2,b_2b_1},
\]
so that $\wedge^2(b_1b_2)=-\wedge^2(b_2b_1)$ and this forces $b_1b_2=\omega^{\pm 2}b_2b_1$. Also, $\phi_{a_i,b_i}^4=\id$, so $\Ad_{b_i}^4=\id$, $i=1,2$. Since $\Ad_{b_1}(b_2)=\omega^{\pm 2}b_2$, we obtain that the order of $\Ad_{b_i}$ is exactly $4$, $i=1,2$, and hence both $\langle \phi_{a_1,b_1},\phi_{a_2,b_2}\rangle$ and $\pi_V\circ\Psi\bigl(\langle \phi_{a_1,b_1},\phi_{a_2,b_2}\rangle\bigr)$ are isomorphic to $\ZZ_4^2$. Therefore, $\pi_V\circ\Psi(Q)$ is necessarily isomorphic to $\ZZ_4^2\times\ZZ_2^2$. Interchanging the indices if necessary, we may always assume that $b_1b_2=\omega^2b_2b_1$.

\begin{lemma}\label{le:4Ibis}
Under the conditions above, $\ker\bigl(\pi_V\circ\Psi\vert_Q\bigr)=\langle \theta\rangle$.
\end{lemma}
\begin{proof}
If $\varphi$ is an element of $Q\cap \ker\bigl(\pi_V\circ\Psi\bigr)$, then there are elements $a\in\SL(U)$ and $\mu\in\FF^\times$ with $\mu^8=1$ (i.e., $\mu 1_V\in\SL(V)$), such that $\varphi=\phi_{a,\mu 1_V}$.

But $\varphi\phi_{a_i,b_i}=\phi_{a_i,b_i}\varphi$ for $i=1,2$, because $Q$ is abelian, so $aa_i=a_ia$ for $i=1,2$ and, since $a_1$ and $a_2$ generate $\End_\FF(U)$, it follows that $a\in\FF^\times 1_U$. Hence $a=\pm 1_U$ and $\varphi\in\ker\Psi=\langle\theta\rangle$.
\end{proof}

\begin{corollary}
$Q$ is isomorphic to $\ZZ_4^3\times \ZZ_2^2$.
\end{corollary}
\begin{proof}
$Q$ is a $2$-group of exponent $4$, with $\pi_V\circ\Psi(Q)\cong\ZZ_4^2\times\ZZ_2^2$ and $\ker\bigl(\pi_V\circ\Psi\vert_Q\bigr)\cong\ZZ_4$.
\end{proof}

Moreover, the quasitorus $\tilde Q=\pi_V\circ\Psi(Q)\,\bigl(\cong \ZZ_4^2\times\ZZ_2^2\bigr)$ induces a division grading on $\End_\FF(V)$.
By Remark \ref{re:division_grading}, $b_1$ and $b_2$ are homogeneous elements, and the arguments in Section \ref{se:division_grading} prove that $\alg\langle b_1,b_2\rangle$ is isomorphic to $M_4(\FF)$, and $\End_\FF(V)=\alg\langle b_1,b_2\rangle\otimes C$, where $C$ is the centralizer in $\End_\FF(V)$ of $\alg\langle b_1,b_2\rangle$. Thus $C$ is a graded subalgebra of $\End_\FF(V)$ isomorphic to $M_2(\FF)$. Hence we may identify $V=V_1\otimes V_2$, $\dim V_1=4$, $\dim V_2=2$, and $\alg\langle b_1,b_2\rangle =\End_\FF(V_1)$, $C=\End_\FF(V_2)$.

Since $C$ is a graded subalgebra, $\tilde Q=\langle [b_1],[b_2],[c_1],[c_2]\rangle$, for elements $c_i\in\SL(V_2)\,\bigl(\subseteq C\bigr)$, $c_i^2=-1_{V_2}$, $i=1,2$, and $c_1c_2=-c_2c_1$. Then there are elements $\hat a_i\in\SL(U)$ such that $\phi_{\hat a_i,c_i}\in Q$, $i=1,2$.

The commutativity of $Q$ gives $\phi_{\hat a_i,c_i}\phi_{a_j,b_j}=\phi_{a_j,b_j}\phi_{\hat a_i,c_i}$ for any $i,j=1,2$, and since $a_1$ and $a_2$ generates $\End_\FF(U)$ and $c_ib_j=b_jc_i$, it follows that $\hat a_i\in\FF^\times 1_U$, so $\hat a_i=\pm 1_U$ ($\det(\hat a_i)=1$). Composing with $\phi_{-1_U,1_V}=\phi_{1_U,\omega^21_V}=\theta^2$ if needed, we get $\phi_{1_U,c_i}\in Q$, $i=1,2$, and hence we obtain the following result:

\begin{theorem}\label{th:4I}
Under the conditions above, the maximal quasitorus $Q$ is isomorphic to $\ZZ_4^3\times\ZZ_2^2$. Moreover, it is given explicitly by
\begin{equation}\label{eq:4I}
Q=\langle \phi_{a_1,b_1},\phi_{a_2,b_2},\phi_{1_U,c_1},\phi_{1_U,c_2},\theta\rangle.
\end{equation}
\end{theorem}

\begin{corollary}\label{co:4I}
Up to conjugation, $\Aut(\fre_8)$ contains a unique maximal finite abelian subgroup with automorphisms of order $4$ and type I.
\end{corollary}

This unique maximal finite abelian subgroup contains too automorphisms of order $4$ corresponding to the diagram:
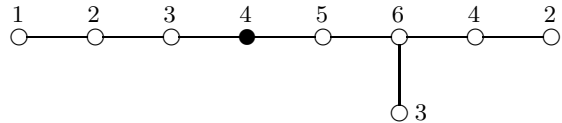
\[
\begin{picture}(80,20)
	\put(0,12){\circle{2}}
	\put(10,12){\circle{2}}
	\put(20,12){\circle{2}}
    \put(30,12){\circle*{2}}
    \put(40,12){\circle{2}}
    \put(50,12){\circle{2}}
    \put(60,12){\circle{2}}
    \put(70,12){\circle{2}}
    \put(50,2){\circle{2}}
	
    \put(1,12){\line(1,0){8}}
    \put(11,12){\line(1,0){8}}
    \put(21,12){\line(1,0){8}}
    \put(31,12){\line(1,0){8}}
    \put(41,12){\line(1,0){8}}
    \put(51,12){\line(1,0){8}}
    \put(61,12){\line(1,0){8}}
    \put(50,3){\line(0,1){8}}
	
	\put(-1,14){\small 1}
	\put(9,14){\small 2}
	\put(19,14){\small 3}
    \put(29,14){\small 4}
    \put(39,14){\small 5}
    \put(49,14){\small 6}
    \put(59,14){\small 4}
    \put(69,14){\small 2}
    \put(52,1){\small 3}
	
\end{picture}
\]
which will be said of \emph{type II}.

\begin{proposition}\label{pr:4I_II}
The maximal abelian quasitorus in Equation \eqref{eq:4I} contains automorphisms of order $4$ and type II:
\end{proposition}
\begin{proof}
Note that the dimension of $\frg^\theta=\frg\subo\simeq \frsl(U)\oplus\frsl(V)$ is $66$, so it is enough to check that the dimension of $\frg^{\phi_{a_1,b_1}}$ is not $66$.

There is a basis of $U$ such that the coordinate matrix of $a_1$ is $\diag(\omega^2,-\omega^2)$, and with $V=V_1\otimes V_2$ as in the arguments leading to Theorem \ref{th:4I}, there are bases of $V_1$ and $V_2$ such that the coordinate matrix of $b_1$ is the Kronecker product $\diag(\omega,\omega^3,\omega^5,\omega^7)\otimes 1_{V_2}$. Now it is easy to compute the dimension of the subspaces of $\frg_{\bar r}$ of the elements fixed by $\phi_{a_1,b_1}$, $\bar r\in\ZZ_4$: $\dim\frg\subo^{\phi_{a_1,b_1}}=1+15=16$, $\dim\frg\subuno^{\phi_{a_1,b_1}}=12=\dim\frg_{\bar 3}^{\phi_{a_1,b_1}}$, and $\dim\frg_{\bar 2}^{\phi_{a_1,b_1}}=20$. Hence $\dim\frg^{\phi_{a_1,b_1}}=60$.
\end{proof}

\bigskip


\section{Order $4$ automorphism. Type II}

In order to deal with the maximal finite abelian subgroups of $\Aut(\fre_8)$ containing automorphisms of order $4$ and type II, we need to recall some results on gradings on simple Lie algebras of type $D_r$, $r\geq 5$ (see \cite[Chapter 3]{EK13}).

Given an abelian group $A$, the $A$-gradings on the orthogonal Lie algebra $\frso(V,q)$ of a vector space $V$ of even dimension $\geq 10$ endowed with a nondegenerate quadratic form $q$ are classified in terms of $A$-gradings on the matrix algebra $\cR=\End_\FF(V)$ compatible with the orthogonal involution $\varphi$ associated to $q$. Note that $\frso(V,q)=\cK(\cR,\varphi)\bydef\{x\in\cR : \varphi(x)=-x\}$. This graded matrix algebra $\cR$ is then isomorphic to $\End_\cD(W)$, where $\cD$ is a graded division algebra and $W$ is a graded right free module for $\cD$. The involution $\varphi$ forces the support $T$ of the grading on $\cD$ to be an elementary $2$-group. Moreover, $\varphi$ is the adjoint relative to a nondegenerate homogeneous hermitian form $B:W\times W\rightarrow \cD$ (i.e., $B$ is $\FF$-bilinear, $B(w_1,w_2)=\tau\bigl(B(w_2,w_1)\bigr)$ and $B(w_1,w_2d)=B(w_1,w_2)d$ for any $w_1,w_2\in W$ and $d\in \cD$, where $\tau$ is a grade-preserving involution on $\cD$). The graded division algebra $\cD$ is then determined by a nondegenerate alternating bicharacter $\beta$ on $T$: $\cD=\espan{X_t: t\in T}$, with $X_tX_s=\beta(t,s)X_sX_t$.

A homogeneous $\cD$-basis $\{v_1,\ldots,v_k\}$ in $W$ can be selected, $\deg v_i=g_i$, such that $B$ is represented by the block-diagonal matrix
\begin{equation}\label{eq:Dr_Phi}
\diag\left(X_{t_1},\ldots,X_{t_q},\matr{0&I\\ I&0},\ldots,\matr{0&I\\ I&0}\right),
\end{equation}
where $I=X_e\in\cD$ and all $X_{t_i}$ are symmetric relative to $\tau$ (see \cite{EldFineClassical} or \cite[Theorem 3.42]{EK13}), with
\begin{equation}\label{eq:Phi_compatibility}
g_1^2t_1=\cdots=g_q^2t_q=g_{q+1}g_{q+2}=\cdots=g_{q+2s-1}g_{q+2s}=g_0^{-1},
\end{equation}
with $q+2s=k$ and $g_0\in A$ is the degree of $B$ (i.e., for any $g,h\in A$, $B(W_g,W_h)\subseteq \cD_{ghg_0}$).

\smallskip

Assume finally that our maximal finite abelian subgroup $Q$ of $G=\Aut(\frg)$, with $\frg=\fre_8$,  contains an automorphism $\theta$ of order $4$ and type II.

The automorphism $\theta$ thus induces a grading by $\ZZ_4$: $\frg=\bigoplus_{\bar r\in\ZZ_4}\frg_{\bar r}$, where $\frg_{\bar r}=\{x\in\frg: \theta(x)=\bi^r x\}$, that satisfies:
\[
\begin{split}
\frg\subo&=\frsl(U)\oplus\frso(V,q),\\
\frg_{\bar 1}&=U\otimes V^+,\\
\frg_{\bar 2}&=\wedge^2 U\otimes  V,\\
\frg_{\bar 3}&=\wedge^3U\otimes V^-,
\end{split}
\]
for a $4$-dimensional vector space $U$ and a $10$-dimensional vector space $V$ endowed with a nondegenerate quadratic form. Here $V^+$ and $V^-$ denote the two half-spin representations of the orthogonal Lie algebra $\frso(V,q)$.

Denote by $C(V,q)$ the Clifford algebra of $(V,q)$, and by $x\cdot y$ the multiplication of any two elements $x,y\in C(V,q)$. Recall that $C(V,q)$ is a unital associative algebra generated by $V$ and that $v^{\cdot 2}=q(v)1$. The Clifford algebra $C(V,q)$ is $\ZZ_2$-graded with $\deg v=\bar 1$ for any $v\in V$. The \emph{spin group} is defined as
\[
\Spin(V,q)\bydef \{x\in C(V,q)\subo^\times : x\cdot V\cdot x^{-1}\subseteq V,\ x\cdot\varsigma(x)=1\},
\]
where $\varsigma$ is the involution (i.e., antiautomorphism of order $2$) of $C(V,q)$ such that $\varsigma(v)=v$ for any $v\in V$.

Let $\{e_1,\ldots,e_{10}\}$ be an orthogonal basis of $V$ with $q(e_i)=1$ for any $i$. Then the center of $C(V,q)\subo$ is $\FF 1\oplus\FF z$, with $z=e_1\cdot e_2\cdot\ldots\cdot e_{10}\in \Spin(V,q)$. Moreover $z^{\cdot 2}=-1$, so the order of $z$ is $4$. There is a surjective homomorphism onto the special orthogonal group:
\[
\begin{split}
\Spin(V,q)&\longrightarrow \SO(V,q)\\
s\quad&\mapsto\quad\iota_s,
\end{split}
\]
where $\iota_s(v)=s\cdot v\cdot s^{-1}=s\cdot v\cdot\varsigma(s)$, for any $v\in V$, whose kernel is $\{\pm 1\}$. Besides, $\iota_z=-1_V$, so the quotient $\Spin(V,q)/\langle z\rangle$ is isomorphic to the projective special orthogonal group $\PSO(V,q)$, which in turn is naturally isomorphic to the connected component $\Aut(\frso(V,q))^{\circ}$.

The Lie algebra $\frso(V,q)$ is isomorphic to the Lie subalgebra
\[
[V,V]^\cdot\bydef \espan{[u,v]^\cdot=u\cdot v-v\cdot u: u,v\in V}
\]
of $C(V,q)\subo^-$. This Lie subalgebra generates $C(V,q)\subo$ (as an associative algebra). The half-spin modules $V^{\pm}$ are the two irreducible modules for the semisimple associative algebra $C(V,q)\subo$ (which are then irreducible modules for $\frso(V,q)\simeq [V,V]^\cdot$). The central element $z$ acts on $V^+$ (respectively $V^-$) by multiplication by the scalar $\bi$ (respectively $-\bi$).

As for the previous cases, we have a homomorphism
\[
\begin{split}
\Phi:\SL(U)\times\Spin(V,q)&\longrightarrow C_G(\theta)\\
(a,s)\qquad&\mapsto\quad \phi_{a,s},
\end{split}
\]
such that $\phi_{a,s}\vert_{\frg\subuno}$ is given by:
\[
\phi_{a,s}(u\otimes x)=a(u)\otimes s. x,
\]
for any $a\in\SL(U)$, $s\in\Spin(V,q)$, $u\in U$ and $x\in V^+$, where $s.x$ denotes the action of the element $s\in C(V,q)\subo$ on $x\in V^+$. Note that $\frg\subuno$ generates $\frg$, so $\phi_{a,s}$ is determined by its action on $\frg\subuno$. The restriction of $\phi_{a,s}$ to $\frg\subo$ is then given by
\[
\phi_{a,s}(b,\sigma)=\bigl(\Ad_a(b),\Ad_s(\sigma)\bigr)
\]
where $\Ad_a(b)=aba^{-1}$ for $a\in \SL(U)$ and $b\in\frsl(U)$, and $\Ad_s(\sigma)=s\cdot\sigma\cdot s^{-1}$ for $s\in\Spin(V,q)$ and $\sigma\in\frso(V,q)\simeq [V,V]^\cdot$ (adjoint action inside $C(V,q)\subo$).
Observe that $\Ad_s(\sigma)=\iota_s\circ\sigma\circ\iota_{s^{-1}}$ for any $s\in \Spin(V,q)$ and $\sigma\in \frso(V,q)\simeq [V,V]^\cdot$. In what follows, the expression $\Ad_s$, for $s\in \Spin(V,q)$, will be used to denote this action of $\Spin(V,q)$ on $\frso(V,q)$.

There is also a group homomorphism:
\[
\begin{split}
\Psi: C_G(\theta)&\longrightarrow \Aut(\frg\subo)\\
\varphi\quad&\mapsto \quad\varphi\vert_{\frg\subo}.
\end{split}
\]

\begin{lemma}\label{le:4II}
\begin{romanenumerate}
\item $\im\Psi=\Aut(\frg\subo)^{\circ}\,\bigl(\simeq\PSL(U)\times\PSO(V,q)\bigr)$ and $\ker\Psi=\langle \theta\rangle$.
\item $\Phi$ is surjective and $\ker\Phi=\langle (-\bi 1_U,z)\rangle$, which is a cyclic group of order $4$.
\end{romanenumerate}
\end{lemma}
\begin{proof}
If $\varphi$ is in $\ker\Psi$, then $\varphi\vert_{\frg\subo}=\id$ and, by Schur's Lemma, $\varphi$ acts as a scalar on $\frg\subuno$. This scalar must be a fourth root of $1$ and hence $\varphi$ is a power of $\theta$.

For any $\varphi\in C_G(\theta)$, the restriction $\varphi\vert_{\frg\subo}$ lies in $\Aut(\frg\subo)^{\circ}\simeq\PSL(U)\times \PSO(V,q)$ because $C_G(\theta)$ is connected (see the proof of Lemma \ref{le:gtheta_semisimple}), so there are elements $a\in \SL(U)$ and $s\in\Spin(V,q)$ such that $\varphi\vert_{\frg\subo}=\phi_{a,s}\vert_{\frg\subo}$. Hence $\varphi\phi_{a,s}^{-1}\in\langle\theta\rangle$. But $\theta=\phi_{\bi 1_U,1}\in\im\Phi$. It follows that $\Phi$ is onto. Also, $\im\Psi=\Psi(\im\Phi)$ fills $\PSL(U)\times \PSO(V,q)\simeq \Aut(\frg\subo)^{\circ}$.

Finally, if $\phi_{a,s}=\id$ for $a\in\SL(U)$ and $s\in\Spin(V,q)$, then $\Ad_a=\id$ and $\Ad_s=\id$, so $a\in\FF^\times 1_U$ and $s\in Z\bigl(C(V,q)\subo\bigr)\cap\Spin(V,q)=\langle z\rangle$. But $\phi_{\lambda 1_U,z^r}=\id$ if and only if $\lambda \bi^r=1$, or $\lambda=(-\bi)^r$. Then $(a,s)=(-\bi 1_U,z)^r$, so $\ker\Phi=\langle (-\bi 1_U,z)\rangle$.
\end{proof}

Denote by $\pi_U$ and $\pi_V$ the projections of $\Aut(\frg\subo)^{\circ}$ onto $\PSL(U)\simeq\Aut(\frsl(U))^{\circ}$ and $\PSO(V,q)\simeq\Aut(\frso(V,q))^{\circ}$ respectively. By Theorem \ref{th:slV}, $\pi_U\circ\Psi(Q)$ is isomorphic to either $\ZZ_4^2$ or $\ZZ_2^4$. Besides, $\pi_U\circ\Psi(Q)\subseteq \PSL(U)\subseteq \Aut(\End_\FF(U))$ induces a division grading on $\End_\FF(U)=\cD$. Therefore, the Brauer invariant of the $\frsl(U)$-module $U$ is $\Br(U)=[\cD]$.

\begin{lemma}\label{le:kerpiVpsi}
The kernel of the restriction of $\pi_V\circ\Psi$ to $Q$ is $\langle \theta\rangle$, so it coincides with the kernel of the restriction of $\Psi$ to $Q$.
\end{lemma}
\begin{proof}
Take $a\in\SL(U)$ and $s\in\Spin(V,q)$ with $\phi_{a,s}\in Q\cap \ker\bigl(\pi_V\circ\Psi\bigr)$. Then $\Ad_s=\id$, so $s\in\langle z\rangle$. But $\phi_{-\bi 1_U,z}=\id$, so if $s=z^r$, $\phi_{a,s}=\phi_{\bi^ra,1}$ and we may assume $s=1$.

Now, $\phi_{a,1}$ is in $Q$ and $Q$ is commutative. But (Theorem \ref{th:slV}) $\pi_U\circ\Psi(Q)$ is either equal to $\langle [x],[y]\rangle$ with $x,y\in\SL(U)$, $x^4=y^4=-1_U$, $xy=\bi yx$, or to $\langle [x_1],[y_1],[x_2],[y_2]\rangle$ with $x_i,y_i\in\SL(U)$, $x_i^2=y_i^2=-1_U$, $x_iy_i=-y_ix_i$, $x_iy_j=y_jx_i$, for $i,j\in\{1,2\}$, $i\ne j$, and $x_1x_2=x_2x_1$, $y_1y_2=y_2y_1$. In the first case (the other case is similar) $\alg\langle x,y\rangle=\End_\FF(U)$ and there are elements $s_x,s_y\in\Spin(V,q)$ such that $\phi_{x,s_x},\phi_{y,s_y}\in Q$. Since $\phi_{a,1}\phi_{x,s_x}=\phi_{x,s_x}\phi_{a,1}$ it follows that $ax=xa$. With the same argument we get $ay=ya$. Hence $a\in Z\bigl(\End_\FF(U)\bigr)\cap\SL(U)=\langle\bi 1_U\rangle$. Therefore $\phi_{a,1}\in\langle \theta=\phi_{\bi 1_U,1}\rangle$.
\end{proof}

The maximal quasitorus $Q$ induces a fine grading $\Gamma$ by the group of characters $A$ of $Q$, which is a refinement of the $\ZZ_4$-grading. Thus $\Gamma$ induces a grading $\Gamma\subo$ on $\frg\subo=\frsl(U)\oplus\frso(V,q)$ by $A$ with support a subgroup of $A$ (the group of characters of $Q/(Q\cap \ker\Psi)=Q/\langle\theta\rangle$). It also induces a grading $\Gamma_{\bar r}$ on each $\frg_{\bar r}$ (an irreducible $\frg\subo$-module), $r=1,2,3$, compatible with $\Gamma\subo$. Then, up to a shift, $\Gamma_{\bar r}$ is the only $A$-grading on $\frg_{\bar r}$ compatible with the grading $\Gamma\subo$. Besides, $\End_\FF(\frg\subuno)\simeq \End_\FF(U)\otimes\End_\FF(V^+)$, and the representation of $\frg\subo$ on $\frg\subuno$ gives a graded homomorphism of associative algebras
\[
\rho: U(\frg\subo)\simeq U(\frsl(U))\otimes U(\frso(V,q))\longrightarrow
\End_\FF(\frg\subuno)\simeq \End_\FF(U)\otimes\End_\FF(V^+).
\]
Then the Brauer invariant of $\frg\subuno$, which is trivial since the grading on $\End_\FF(\frg\subuno)$ is elementary, satisfies
%
%
\[
1=\Br(U)\Br(V^+)=[\cD]\Br(V^+),
\]
and hence $\Br(V^+)=[\cD]^{-1}=[\cD^\text{op}]$.

On the other hand, working with the grading $\Gamma_{\bar 2}$ on $\frg_{\bar 2}$ we get (using \cite[Theorem 15]{EKpr})
\[
1=\Br(\wedge^2U)\Br(V)=[\cD]^2\Br(V).
\]
We are left with two cases.

\medskip

\subsection{Case a)}
If $\pi_U\circ\Psi(Q)$ is isomorphic to $\ZZ_2^4$, then $\cD=\alg\langle x_1,y_1,x_2,y_2\rangle$ with $x_i^2=y_i^2=-1_U$, $x_iy_i=-y_ix_i$, $x_iy_j=y_jx_i$, for $i,j\in\{1,2\}$, $i\ne j$, and $x_1x_2=x_2x_1$, $y_1y_2=y_2y_1$. That is, $\cD$
is a twisted group algebra $\FF^\sigma T$, where $T\simeq \ZZ_2^4$, where $X_{(\bar 1,\bar 0,\bar 0,\bar 0)}=x_1$, $X_{(\bar 0,\bar 1,\bar 0,\bar 0)}=y_1$, $X_{(\bar 0,\bar 0,\bar 1,\bar 0)}=x_2$ and $X_{(\bar 0,\bar 0,\bar 0,\bar 1)}=y_2$. Here $[\cD]=[\cD]^{-1}$.

In this case $\Br(V)=[\cD]^2=1$, and this means that the grading induced in $\frso(V,q)$ is elementary, i.e., induced by a grading on $V$. Since the neutral component is trivial: $\frso(V,q)_e=0$, there is an orthogonal homogeneous basis $\{e_1,\ldots,e_{10}\}$ of $V$, with $q(e_i)=1$ and $\deg(e_i)=g_i$ for any $i$, with $g_i^2=e$ for any $i$ (recall Equations \eqref{eq:Dr_Phi} and \eqref{eq:Phi_compatibility}, and we can adjust in this case so as to get $g_0=e$ by shifting the grading on $V$). Also, the condition $\frso(V,q)_e=0$ forces all the $g_i$'s to be different. Since $\pi_V\circ\Psi(Q)$ is contained in $\PSO(V,q)$, we have $g_1g_2\cdots g_{10}=e$ (see \cite[Lemma~33]{EKpr}).

The elements of $\pi_V\circ\Psi(Q)\,\bigl(\subseteq \PSO(V,q)\simeq \Aut(\frso(V,q))^\circ\bigr)$ lift then to $\SO(V,q)$ and they act diagonally in this basis with eigenvalues $1$ or $-1$. Therefore we have
\[
\pi_V\circ\Psi(Q)\subseteq \langle \Ad_{e_i\cdot e_j}: i\ne j\rangle,
\]
because $\iota_{e_i\cdot e_j}$ is the diagonal endomorphism with $e_i$ and $e_j$ eigenvectors with eigenvalue $-1$, and $e_h$ is fixed by $\iota_{e_i\cdot e_j}$ for $h\ne i,j$.

\begin{remark}
The map
\[
\begin{split}
\ZZ_2^9&\longrightarrow \langle \Ad_{e_i\cdot e_j}: i\ne j\rangle\\
\epsilon_i&\mapsto\quad \Ad_{e_i\cdot e_{i+1}},
\end{split}
\]
where $\epsilon_i=(\bar 0,\ldots,\bar 1,\ldots,\bar 0)$ ($\bar 1$ in the $i^\text{th}$ position), is a surjective homomorphism with kernel $\langle (\bar 1,\ldots,\bar 1)\rangle=\langle \epsilon_1+\cdots+\epsilon_9\rangle$, because $e_1\cdot e_2\cdot\ldots\cdot e_{10}=z$ and $\Ad_z=\id$. Therefore, the group $\langle \Ad_{e_i\cdot e_j}: i\ne j\rangle$ is isomorphic to $\ZZ_2^8$.
Hence $\pi_V\circ\Psi(Q)$ is isomorphic to $\ZZ_2^m$ with $m\leq 8$, and hence $Q$ is isomorphic to $\ZZ_4\times\ZZ_2^m$ ($4\leq m\leq 8$). We will see that $m=6$.

Besides, $e_1\cdot e_2\cdot\ldots\cdot e_8=-z\cdot e_9\cdot e_{10}$, so $\Ad_{e_1\cdot e_2\cdot\ldots\cdot e_8}=\Ad_{e_9\cdot e_{10}}$. In the same vein, $e_1\cdot e_2\cdot\ldots\cdot e_6=z\cdot e_7\cdot e_8\cdot e_9\cdot e_{10}$. Thus the group $\langle \Ad_{e_i\cdot e_j}: i\ne j\rangle$ consists of the elements $\id$, $\Ad_{e_i\cdot e_j}$ for $i\ne j$, and $\Ad_{e_i\cdot e_j\cdot e_k\cdot e_l}$ for different $i,j,k,l$. \qed
\end{remark}

\begin{lemma}\label{le:Xts}
For any element $\sigma\in\pi_V\circ\Psi(Q)$ there exists a unique element $t\in T$ and a unique class $s\langle z\rangle\in \Spin(V,q)/\langle z\rangle$ such that $\sigma=\pi_V\circ\Psi(\phi_{X_t,s})$.
\end{lemma}
\begin{proof}
Since $\pi_U\circ\Psi(Q)=\{[X_t]: t\in T\}\subseteq \PSL(U)$ by Remark \ref{re:Xt}, for any $\sigma\in\pi_V\circ\Psi(Q)$ there is an element $t\in T$ and an element $s\in\Spin(V,q)$ such that $\sigma=\phi_{X_t,s}\vert_{\frso(V,q)}$. If $\phi_{X_t,s}\vert_{\frso(V,q)}=\phi_{X_{t'},s'}\vert_{\frso(V,q)}$, then $\phi_{X_tX_{t'}^{-1},s(s')^{-1}}$ is in $Q\cap\ker\bigl(\pi_V\circ\Psi\bigr)=\langle\theta\rangle$ (Lemma \ref{le:kerpiVpsi}). Hence we obtain $(X_tX_{t'}^{-1},s(s')^{-1})\in\Phi^{-1}(\langle\theta\rangle)=\langle (-\bi 1_U,z),(1_U,z)\rangle$. It follows that $X_t\in\FF^\times X_{t'}$, so $t=t'$ and $s(s')^{-1}\in\langle z\rangle$.
\end{proof}

This shows that $Q$ is determined by $\pi_V\circ\Psi(Q)$ and $\theta=\phi_{1_U,z}$. In other words, the grading on $\frso(V,q)$ and the $\ZZ_4$-grading determine the fine grading $\Gamma$ on $\frg$ induced by $Q$.

\smallskip

We are going to reduce the problem of determining $Q$ to an easy combinatorial problem.

Since $\pi_U\circ\Psi(Q)=\langle [x_1],[y_1],[x_2],[y_2]\rangle$, there are elements $p_1,q_1,p_2,q_2\in\Spin(V,q)$ with $\phi_{x_1,p_1},\phi_{y_1,q_1},\phi_{x_2,p_2},\phi_{y_2,q_2}\in Q$. The commutativity of $Q$ gives $\phi_{x_1,p_1}\phi_{y_1,q_1}=\phi_{y_1,q_1}\phi_{x_1,p_1}$, so $\phi_{x_1y_1,p_1\cdot q_1}=\phi_{y_1x_1,q_1\cdot p_1}=\phi_{-x_1y_1,q_1\cdot p_1}=\phi_{x_1y_1,-q_1\cdot p_1}$ and we get $p_1\cdot q_1=-q_1\cdot p_1$. In this vein, we check that the elements $p_1,q_1,p_2,q_2$ in $\Spin(V,q)\subseteq C(V,q)\subo$ satisfy the same commutation relations as $x_1,y_1,x_2,y_2$. In particular the elements $p_i,q_i$ are not in $Z(\Spin(V,q))=\langle z\rangle$. Besides, $\Ad_{p_1}\in\langle \Ad_{e_i\cdot e_j}: i\ne j\rangle$, and $\theta\phi_{x_1,p_1}=\phi_{1_U,z}\phi_{x_1,p_1}=\phi_{x_1,z\cdot p_1}$, $\theta^2\phi_{x_1,p_1}=\phi_{x_1,-p_1}$. Hence we may replace $p_1$ by $z\cdot p_1$ or by $-p_1$, and the same for $p_2,q_1,q_2$. Therefore, we may assume that each $p_i$ or $q_i$ is of the form $e_i\cdot e_j$ for $i\ne j$, or $e_i\cdot e_j\cdot e_k\cdot e_l$, for different $i,j,k,l$.

Then $Q$ is generated by $\phi_{x_1,p_1}$, $\phi_{y_1,q_1}$, $\phi_{x_2,p_2}$, $\phi_{y_2,q_2}$ and by the elements in $Q$ of the form $\phi_{1_U,s}$ with $s\in\Spin(V,q)$. Moreover, these elements $s$ belong to $\{\pm e_I: \lvert I\rvert\ \text{even}\}$, where for any sequence $I=(i_1,\ldots,i_r)$ with $1\leq i_1<\cdots<i_r\leq 10$, $e_I$ denotes the element $e_{i_1}\cdot e_{i_2}\cdot\ldots\cdot e_{i_r}\in C(V,q)$. ($e_I$ belongs to $\Spin(V,q)$ if and only if $\lvert I\rvert =r$ is even.)
Besides, the commutativity of $Q$ forces that any element $s\in \Spin(V,q)$ such that $\phi_{1_U,s}\in Q$ commutes with  $p_1,q_1,p_2,q_2$, and any two elements $s,s'\in\Spin(V,q)$ with $\phi_{1_U,s},\phi_{1_U,s'}\in Q$ commute.

\begin{lemma}\label{le:No_phi_1eiej}
$Q$ does not contain any element of the form $\phi_{1_U,e_i\cdot e_j}$ with $1\leq i\ne j\leq 10$.
\end{lemma}
\begin{proof}
The automorphism $\phi_{1_U,e_i\cdot e_j}$ has order $4$ and fixes elementwise a subalgebra isomorphic to $\frsl(U)\oplus \frso_2\oplus\frso_8$ inside $\frg\subo$, and a subspace of the form $\wedge^2U\otimes W$, $\dim W=8$ in $\frg_{\bar 2}$. Hence $\dim\frg^{\phi_{1_U,e_i\cdot e_j}}\geq (15+1+28)+\binom{4}{2}\times 8=92$. But any order $4$ element in $Q$ is an automorphism of type I or II, whose subalgebras of fixed elements have dimension $66$ and $60$ respectively. Therefore $\phi_{1_U,e_i\cdot e_j}$ does not belong to $Q$ for any $i\ne j$.
\end{proof}

\begin{remark}\label{re:No_phi_1eiej}
Since $\theta\phi_{1_U,e_i\cdot e_j}=\phi_{1_U,z\cdot e_i\cdot e_j}$ and $\theta^2\phi_{1_U,e_i\cdot e_j}=\phi_{1_U,-e_i\cdot e_j}$, our maximal quasitorus $Q$ does not contain elements of the form $\phi_{1_U,\pm e_i\cdot e_j}$ or $\phi_{1_U,\pm e_{i_1}\cdot\ldots\cdot e_{i_8}}$ ($1\leq i_1<\cdots <i_8\leq 10$).
\end{remark}

Given two sequences $I=(i_1,\ldots,i_r)$ and $J=(j_1,\ldots,j_s)$, with $1\leq i_1<\cdots <i_r\leq 10$, $1\leq j_1<\cdots<j_s\leq 10$, consider the elements $e_I\bydef e_{i_1}\cdot e_{i_2}\cdot\ldots\cdot e_{i_r}$, $e_J\bydef e_{j_1}\cdot e_{j_2}\cdot\ldots\cdot e_{j_s}$ in $C(V,q)$ as above. Then $e_I\cdot e_J=(-1)^{\lvert I\cap J\rvert}e_J\cdot e_I$.

Identify $I$ with the element $x_I\in\ZZ_2^{10}$ with $\bar 1$'s in the components $i_1,\ldots, i_r$ and $\bar 0$'s elsewhere, and similarly for $J$. Then $e_I\cdot e_J=(-1)^{x_I\bullet x_J}e_J\cdot e_I$, where for elements $x,y\in\ZZ_2^{10}$, $x\bullet y=\sum_{i=1}^{10}x_iy_i$ denotes the natural symmetric nondegenerate bilinear form on $\ZZ_2^{10}$. In other words, the elements $e_I$ and $e_J$ commute in $C(V,q)$ if and only if $x_I$ and $x_J$ are orthogonal in $\ZZ_2^{10}$.

Note that $z=e_1\cdot e_2\cdot\ldots\cdot e_{10}=e_{(1,2,3,\ldots,10)}$. Let $K=\ZZ_2(\bar 1,\bar 1,\ldots,\bar 1)$, then $K^\perp$ is the subspace of the elements $x_I$ with $\lvert I\rvert$ even (i.e., $e_I\in\Spin(V,q)$). The bilinear form $\bullet$ induces a nondegenerate alternating bilinear form on $K^\perp/K$.

Then the problem of finding a maximal finite abelian subgroup $Q$ under the conditions above is equivalent to the problem of finding maximal subspaces $S$ of $K^\perp/K$ which are the orthogonal sum of two orthogonal hyperbolic planes $U_1$ and $U_2$ (corresponding to $\{p_1,q_1\}$ and $\{p_2,q_2\}$), and a totally isotropic subspace $U_3$ orthogonal to $U_1$ and $U_2$. That is, $S=U_1\perp U_2\perp U_3$, $U_i\bullet U_i\ne 0$, $i=1,2$, $U_3\bullet U_3=0$, $U_i\bullet U_j=0$ for $i\ne j$, and $\dim_{\ZZ_2}U_1=\dim_{\ZZ_2}U_2=2$, and with the extra condition that there is no $x+K\in U_3$ with $\lvert\supp(x)\rvert=2$ or $8$, where $\supp(x)$ is the set of indices with $x_i\ne \bar 0$ (because of Remark \ref{re:No_phi_1eiej}).

Since $\dim_{\ZZ_2}K^\perp/K=8$, the maximal dimension for $U_3$ is $2$ and in this case $S=U_3^\perp$. As above, let $\{\epsilon_1,\ldots,\epsilon_{10}\}$ be the canonical basis of $\ZZ_2^{10}$, and denote by $\bar x$ the class modulo $K$ of an element $x\in \ZZ_2^{10}$. Up to reordering of indices, the only `maximal' possibility is given by
\[
U_3=\ZZ_2(\overline{\epsilon_1+\epsilon_2+\epsilon_3+\epsilon_4})\oplus
\ZZ_2(\overline{\epsilon_3+\epsilon_4+\epsilon_5+\epsilon_6}).
\]
(Note that $\ZZ_2(\overline{\epsilon_1+\epsilon_2+\epsilon_3+\epsilon_4})\oplus
\ZZ_2(\overline{\epsilon_5+\epsilon_6+\epsilon_7+\epsilon_8})$ is not valid as $\overline{\epsilon_1+\cdots+\epsilon_8}=\overline{\epsilon_9+\epsilon_{10}}$, which is the class, modulo $K$, of an element with support of size $2$.)

Then, up to a reordering of indices, there is a unique maximal possibility:
\[
S=U_3^\perp=\espan{\overline{\epsilon_7+\epsilon_8},
\overline{\epsilon_8+\epsilon_9},
\overline{\epsilon_9+\epsilon_{10}},
\overline{\epsilon_1+\epsilon_2},
\overline{\epsilon_3+\epsilon_4},
\overline{\epsilon_5+\epsilon_6}}.
\]
Therefore, we obtain the main result of this \textbf{Case a)}:

\begin{theorem}\label{th:4II}
Up to conjugation, there is a unique maximal finite abelian subgroup $Q$ of $\Aut(\fre_8)$ with the properties that it contains an automorphism $\theta$ of order $4$ of type II, and the image of $Q$ in the automorphism group of the simple ideal of $(\fre_8)^\theta$ isomorphic to $\frsl_4(\FF)$ is $2$-elementary abelian. This maximal quasitorus is isomorphic to $\ZZ_4\times\ZZ_2^6$.
\end{theorem}

A concrete realization of this maximal quasitorus is the subgroup generated by the elements $\phi_{x_1,e_8\cdot e_9}$, $\phi_{y_1,e_9\cdot e_{10}}$, $\phi_{x_2,e_3\cdot e_4}$, $\phi_{y_2,e_1\cdot e_3\cdot e_6\cdot e_7}$, $\phi_{1_U,e_1\cdot e_2\cdot e_3\cdot e_4}$, $\phi_{1_U,e_3\cdot e_4\cdot e_5\cdot e_6}$ and $\theta=\phi_{1_U,z}$. This is the cartesian product of the cyclic subgroup $\langle \theta\rangle$ of order $4$  and the cyclic subgroups generated by the order $2$ elements $\phi_{x_1,e_8\cdot e_9}$, $\phi_{y_1,e_9\cdot e_{10}}$, $\phi_{x_2,e_3\cdot e_4}$, $\phi_{y_2,z\cdot e_1\cdot e_3\cdot e_6\cdot e_7}$, $\phi_{1_U,e_1\cdot e_2\cdot e_3\cdot e_4}$, and $\phi_{1_U,e_3\cdot e_4\cdot e_5\cdot e_6}$.

\medskip

\subsection{Case b)}
If $\pi_U\circ\Psi(Q)$ is isomorphic to $\ZZ_4^2$, then $\cD=\alg\langle x,y\rangle$, with $x^4=y^4=-1$, $xy=\bi yx$, and this is a twisted algebra $\FF^\sigma T$ with $T\simeq \ZZ_4^2$. Hence, $[\cD]$ corresponds to the pair $(T,\beta)$, with $T=\ZZ_4^2$ and $\beta:T\times T\rightarrow \FF^\times$ the alternating bicharacter such that $\beta\bigl((\bar 1,\bar 0),(\bar 0,\bar 1)\bigr)=\bi$. In this case $[\cD]^4=1$ and $\Br(V)=[\cD]^{-2}=[\cD]^2$ corresponds to the pair $(\tilde T,\tilde \beta)$, with $\tilde T=\ZZ_2^2\,\bigl(\simeq (2\ZZ_4)^2\bigr)$, and $\tilde\beta:\tilde T\times\tilde T\rightarrow\FF^\times$ such that $\tilde\beta\bigl((\bar 1,\bar 0),(\bar 0,\bar 1)\bigr)=-1$.

The classification of gradings in the simple Lie algebras of type $D$ \cite[\S 3.6]{EK13} shows that the restriction of the grading $\Gamma$ induced by $Q$ (this is a grading by the group of characters $A$ of $Q$) to $\frso(V,q)$, which is a grading with trivial neutral component, is up to equivalence a grading obtained as follows. There is a subgroup $\tilde T$ isomorphic to $\ZZ_2^2$ of the group $A$ of characters of $Q$, so $\tilde T=\langle a,b\rangle$, with $a$ and $b$ of order $2$ and $\tilde\beta(a,b)=-1$. Consider the graded division algebra $\tilde\cD\cong M_2(\FF)$ generated by $X_a,X_b\in\SL_2(\FF)$, with $X_a^2=X_b^2=-1$ and $X_aX_b=\tilde\beta(a,b)X_bX_a=-X_bX_a$. Let $c=ab$ and $X_c\bydef X_aX_b$. Let $\tau$ be the involution on $\tilde\cD$ which fixes $X_a$ and $X_b$ (orthogonal involution). Finally, let $W$ be a free graded right $\tilde\cD$-module of dimension $5$ endowed with a nondegenerate homogeneous hermitian form $B:W\times W\rightarrow \tilde\cD$. A homogeneous basis $\{v_1,\ldots,v_5\}$ of $W$ as a right $\tilde\cD$-module can be selected, $\deg v_i=g_i$, such that the coordinate matrix of $B$ is (see Equation~\eqref{eq:Dr_Phi})
\[
\diag\bigl(X_{t_1},X_{t_2},X_{t_3},X_{t_4},X_{t_5}\bigr),
\]
with $t_1,\ldots,t_5\in\{e,a,b\}$ and with the extra condition (Equation \eqref{eq:Phi_compatibility}):
\begin{equation}\label{eq:g1g5}
g_1^2t_1=\cdots= g_5^2t_5.
\end{equation}
Moreover, the condition on the neutral component being trivial shows that the classes of $g_1,\ldots, g_5$ modulo $\tilde T$ are different.

Then $\frsu(W,B)\bydef\{x\in\End_{\tilde\cD}(W): B(xw_1,w_2)+B(w_1,xw_2)=0\ \forall w_1,w_2\in W\}$ is isomorphic to $\frso(V,q)$. The grading on $W$ induces a grading on $\frsu(W,B)$ and hence on $\frso(V,q)$.

Up to equivalence, we may assume (\cite[\S 3.6]{EK13} or \cite{EldFineClassical}) that $(t_1,\ldots,t_5)$ is one of the following:
\[
(e,e,e,e,e),\ (e,e,e,e,a),\ (e,e,e,a,a),\ (e,e,e,a,b),\ (e,e,a,a,b).
\]
The condition on $\pi_V\circ\Psi(Q)$ lying on $\Aut(\frso(V,q))^{\circ}\simeq \PSO(V,q)$, is equivalent to $ct_1t_2t_3t_4t_5=e$ \cite[Lemma~33]{EKpr}, which is only satisfied for $(e,e,e,a,b)$. We also may shift the grading on $W$ without changing the grading on $\frsu(W,B)$, so we may assume $g_1=e$. Therefore the support of the grading on $\frso(V,q)$, which by Lemma \ref{le:kerpiVpsi} coincides with the support of $\Gamma\subo$, is generated by the elements $g_2$, $g_3$, $g_4$, $g_5$ and $\tilde T$, with $e=g_2^2=g_3^2=g_4^2a=g_5^2b$ by \eqref{eq:g1g5}, and hence the grading subgroup of $\Gamma\subo$, which is the subgroup of $A$ given by the group of characters of $Q/Q\cap\ker\bigl(\pi_V\circ\Psi\bigr) =Q/\langle\theta\rangle$ is a quotient of $\ZZ_4^2\times\ZZ_2^2$, and it contains a copy of $\ZZ_4^2$ because $\pi_U\circ\Psi(Q)\cong\ZZ_4^2$. In particular, the grading on $\frso(V,q)$ is a coarsening of a unique fine grading with universal group $\ZZ_4^2\times\ZZ_2^2$ and hence $Q/\langle\theta\rangle$ is isomorphic to $\ZZ_4^2\times\ZZ_2^r$, with $r\leq 2$, so that $Q$ is isomorphic to $\ZZ_4^3\times\ZZ_2^r$.

But we know that there is a maximal abelian subgroup $\tilde Q$ of $\Aut(\frg)$ isomorphic to $\ZZ_4^3\times\ZZ_2^2$ (the case of the existence of automorphisms of order $4$ and type I) and this maximal subgroup contains automorphisms of order $4$ and type II (Proposition \ref{pr:4I_II}). By maximality, this is the only possibility, up to equivalence, so no new maximal finite abelian subgroups of $\Aut(\frg)$ appear in this subcase.

This finishes the proof of Theorem \ref{th:main_result}.

\begin{remark}
In this Case b), the grading on $\frsl(U)$ is determined by the grading on $\frso(V,q)$ and the condition that the Brauer invariant of $V^+$ is $[\End_\FF(U)]^{-1}$, with a division grading by $\ZZ_4^2$ on $\End_\FF(U)$. Alternatively, an analogue of Lemma \ref{le:Xts} works here, and this shows that $Q$ is determined by $\pi_V\circ\Psi(Q)$.
\end{remark}

\bigskip

\section{Models of the gradings}

Most of the fine gradings induced by the maximal finite abelian subgroups of $\Aut(\fre_8)$ have nice descriptions in terms of some nonassociative algebras and constructions of $\fre_8$ related to them. Here we will briefly review some of these models and point to suitable references.

Let $\cC$ be the Cayley algebra over our ground field $\FF$. The algebra $\cC$ is not associative, although any two elements generate an associative subalgebra. It can be obtained through the Cayley-Dickson doubling process in three steps, starting with $\FF$, in the same way as the classical real division algebra of the octonions is obtained from the reals:
\[
\RR\ \leq\ \CC=\RR\oplus\RR\bi\ \leq\ \HH=\CC\oplus\CC\bj\ \leq\ \OO=\HH\oplus\HH\bl.
\]
Each step gives a grading by $\ZZ_2$, and hence $\cC$ is graded by $\ZZ_2^3$ with all the homogeneous components of dimension $1$ (see \cite{Eld_GradingsOctonions} or \cite[Chapter~4]{EK13}).

The algebra $\cJ=\cH_3(\cC)$ of hermitian $3\times 3$-matrices over $\cC$, with the symmetrized product $X\circ Y=\frac{1}{2}(XY+YX)$, is  the only exceptional simple Jordan algebra, up to isomorphism. Among the fine gradings on $\cJ$ (see \cite{DM_F4} or \cite[Chapter~5]{EK13}), there is a fine grading by $\ZZ_3^3$ with all the homogeneous components of dimension $1$.

Then the Lie algebra obtained by means of a construction by Tits \cite{Tits66} on the vector space
\[
\cT(\cC,\cJ)=\Der(\cC)\oplus\bigl(\cC_0\otimes\cJ_0\bigr)\oplus\Der(\cJ),
\]
where $\cC_0$ (respectively $\cJ_0$) is the subspace of trace zero elements in $\cC$ (resp. $\cJ$), with suitable Lie bracket (see \cite[(6.14)]{EK13}), is the simple Lie algebra $\fre_8$. The Lie algebras of derivations $\Der(\cC)$ and $\Der(\cJ)$ are commuting subalgebras of $\cT(\cC,\cJ)$, and the fine gradings by $\ZZ_2^3$ on $\cC$ and by $\ZZ_3^3$ on $\cJ$ combine to give the fine grading by $\ZZ_2^3\times\ZZ_3^3\simeq \ZZ_6^3$ on $\fre_8$ in Theorem \ref{th:6}. The corresponding quasitorus appears also, in a completely different way, in \cite{Vogan}.

\smallskip

Now, let $\cQ=M_2(\FF)$ be the `quaternion algebra' over $\FF$, endowed with its natural (symplectic) involution. The algebra $\cH_4(\cQ)$ of hermitian $4\times 4$-matrices over $\cQ$ is a simple Jordan algebra. There is a Cayley-Dickson doubling process \cite{AF84}, similar to the one used to get $\cC$, that gives an algebra with involution on the direct sum of two copies of $\cH_4(\cQ)$:
\begin{equation}\label{eq:A}
\cA=\cH_4(\cQ)\oplus v\cH_4(\cQ).
\end{equation}
This algebra is a simple \emph{structurable} algebra of dimension $56$. The structurable algebras form a class of algebras with involution introduced in \cite{All78} that contains the classes of associative algebras with involutions and of Jordan algebras (with trivial involution).

Given any structurable algebra $\cX$ with involution $x\mapsto \bar x$, the \emph{Steinberg unitary Lie algebra} $\frstu_3(\cX)$ (see \cite{AF93}) is defined as the Lie algebra generated by symbols $u_{ij}(x)$, $1\leq i\ne j\leq 3$, $x\in\cX$, subject to the relations:
\[
\begin{split}
&u_{ij}(x)=u_{ji}(-\bar x),\\
&x\mapsto u_{ij}(x)\ \text{is linear},\\
&[u_{ij}(x),u_{jk}(y)]=u_{ik}(xy)\ \text{for distinct $i,j,k$.}
\end{split}
\]
There is a decomposition $\frstu_3(\cX)=\frs\oplus u_{12}(\cX)\oplus u_{23}(\cX)\oplus u_{31}(\cX)$, with $\frs=\sum_{i<j}[u_{ij}(\cX),u_{ij}(\cX)]$, which is a grading by $\ZZ_2^2$. Any grading by a group $G$ on the structurable algebra $\cX$ (compatible with the involution) induces naturally a grading by $\ZZ_2^2\times G$ on $\frstu_3(\cX)$. For the structurable algebra $\cA$ in \eqref{eq:A}, the Steinberg unitary Lie algebra $\frstu_3(\cA)$ is the simple Lie algebra of type $E_8$.

Moreover, $\cA$ is endowed with a $\ZZ_4^3$-grading where all homogeneous components have dimension at most $1$ (see \cite{AEKpr}). In this way, we obtain a fine grading by $\ZZ_2^2\times\ZZ_4^3$ on $\fre_8\simeq\frstu_3(\cA)$ which necessarily is, up to equivalence, the grading induced by the quasitorus in Theorem \ref{th:4I}.

Also, the Jordan algebra $\cH_4(\cQ)$ has a fine $\ZZ_4\times\ZZ_2^3$-grading (see \cite[\S 5.6]{EK13}), and this induces a fine grading by $\ZZ_4\times\ZZ_2^4$ on the structurable algebra $\cA$ in \eqref{eq:A}, and hence by $\ZZ_4\times\ZZ_2^6$ on $\fre_8\simeq\frstu_3(\cA)$. This is the fine grading in Theorem \ref{th:4II}.

\smallskip

As for the fine gradings by elementary abelian groups:
\begin{itemize}
\item The fine $\ZZ_2^8$-grading is induced on $\fre_8=\cT(\cC,\cJ)$ by the fine $\ZZ_2^3$-grading on $\cC$ and a fine $\ZZ_2^5$-grading on $\cJ$. Alternatively it is obtained by means of the natural $\ZZ_2^6$-grading on $\cC\otimes\cC$, which is a structurable algebra so it induces a grading by $\ZZ_2^2\times\ZZ_2^6=\ZZ_2^8$ on $\frstu_3(\cC\otimes\cC)$, and this is again isomorphic to $\fre_8$.

\item The Jordan algebra $\cH_4(\cQ)$ is also endowed with a fine $\ZZ_2^6$-grading that gives a fine $\ZZ_2^7$-grading on $\cA$ and hence a fine $\ZZ_2^9$-grading on $\frstu_3(\cA)=\fre_8$.

\item The \emph{Okubo algebra} $\cO$ is an eight-dimensional symmetric composition algebra endowed with a fine $\ZZ_3^2$-grading. The Lie algebra $\fre_8$ can be constructed on the vector space
    \[
    \bigl(\tri(\cO)\oplus\tri(\cO)\bigr)\oplus\bigl(\bigoplus_{i=1}^3\iota_i(\cO\otimes\cO)\bigr)
    \]
    with a natural bracket, where each $\iota_i(\cO\otimes\cO)$ is a copy of the tensor product $\cO\otimes\cO$ and where
    \[
    \tri(\cO)=\{(d_1,d_2,d_3)\in\frso(\cO): d_1(xy)=d_2(x)y+xd_3(y)\ \forall x,y\}
    \]
    is the \emph{triality Lie algebra} of $\cO$ (with componentwise bracket). There is a natural order $3$ automorphism $\Theta$ that permutes cyclically the three copies of $\cO\otimes\cO$ and the components in each copy of $\tri(\cO)$. This automorphism $\Theta$ induces a $\ZZ_3$-grading compatible with the $\ZZ_3^4$-grading on $\cO\otimes\cO$ obtained by combining the $\ZZ_3^2$-grading on each copy of $\cO$. In this way we obtain the fine $\ZZ_3^5$-grading on $\fre_8$. (See \cite{Eld_gradings_symmetric} and the references therein for details.)

\item Finally, the fine $\ZZ_5^3$-grading is an example of a \emph{Jordan grading} (the maximal quasitorus is a \emph{Jordan subgroup} \cite{Alekse}). A specific model for this grading, which does not rely on nonassociative algebras, is given in \cite{Eld_Jordan}.
\end{itemize}



\end{document}